\documentclass[preprint,12pt]{elsarticle}

\usepackage{amsmath,bm}
 \usepackage{amsthm}
 \usepackage{amsfonts,amssymb}
 \usepackage{setspace,color}
 \usepackage{cases}
 \usepackage{extarrows}
\usepackage{hyperref}
\allowdisplaybreaks

\numberwithin{equation}{section}

\newtheorem{thm}{Theorem}[section]
\newtheorem{lem}{Lemma}[section]
\newtheorem{claim}{Claim}[section]

\newtheorem{prop}{Proposition}[section]
\theoremstyle{definition}
\newtheorem{defn}{Definition}[section]
\theoremstyle{remark}
\newtheorem{rem}{Remark}[section]
\numberwithin{equation}{section}
\journal{NARWA}
\begin{document}

\begin{frontmatter}



\title{Global existence and asymptotics for the modified two-dimensional Schr\"{o}dinger equation in the critical regime}

\author{Xuan Liu}
\ead{lxmath@zju.edu.cn}
\author{Ting Zhang\corref{*}}
\ead{zhangting79@zju.edu.cn}
\address{School of Mathematical Sciences, Zhejiang University,  Hangzhou 310027,  China}
\cortext[*]{Corresponding author.}
\begin{abstract}
 We study the asymptotic behavior of  the  modified two-dimensional Schr\"{o}dinger equation
		$ (D_t -F(D))u=\lambda|u| u$ in the critical regime,
		where $\lambda \in \mathbb{C}$ with $\text{Im} \lambda \ge0$  and  $F(\xi)$ is a second order constant coefficients  elliptic symbol. For any
		smooth initial datum  of size $\varepsilon\ll1$, we prove that the solution is  global-in-time, combining the vector fields method and a semiclassical analysis method introduced by Delort.   Moreover, we present the pointwise decay estimates and the large time asymptotic formulas of the solution.
\end{abstract}


\end{frontmatter}


\section{Introduction}\label{s1}
We consider the following   modified  critical nonlinear Schr\"{o}dinger equation
\begin{equation}
		(D_t -F(D))u=\lambda|u|^{\frac{2}{n}} u,\ x\in\mathbb{R}^n \label{NLS0}
\end{equation}
where $D_t=\frac{\partial_t}{i}$, $D=\frac{\nabla_x}{i}$, $\lambda \in \mathbb{C}$ and  $F(D)$ is defined via its real symbol, i.e.,
\begin{equation}
	F(D)u(x)=\frac{1}{(2\pi)^{n/2}} \int _{\mathbb{R}^n} e^{ix\cdot \xi} F(\xi) (\mathcal{F} u)(\xi)d\xi,\notag
\end{equation}
where  $\mathcal{F} $ is the Fourier transform with respect to  $x$ variables. The  Cauchy problem (\ref{NLS0}) is critical because the best time decay one can expect for the solution of the linear equation $(D_t-F(D))u=0$ with smooth, decaying Cauchy data is $\|u(t,\cdot )\|_{L^\infty }=O(t^{-n/2})$, so that the nonlinearity will satisfy $\||u|^{\frac{2}{n}}u(t,\cdot )\|_{L^2}\le Ct^{-1}\|u(t,\cdot )\|_{L^2}$, with a time factor $t^{-1}$ just at the limit of integrability. Schr\"odinger equations with modified dispersion contain many important equations from  the fields of physics. Typical models are the Schr\"odinger equation($F(\xi)=\xi^2$), KDV equation($F(\xi)=\xi^3$), Klein-Gordon equation($F(\xi)=\sqrt {1+\xi^2}$) and Benjamin-Ono equation($F(\xi)=\xi |\xi|$) etc.  Over the past decades, the local smoothing properties, the dispersive estimates and the well-posedness of the modified Schr\"{o}dinger equations have been studied extensively, see e.g. \cite{Balabane,Ben,Carles,CS,HHZ,Hoshiro,Kenig,KPV1,KPV2} etc. Then we want to study the global existence and asymptotics of solutions to the  modified Schr\"{o}dinger equation (\ref{NLS0}), provided that the initial data is small and spatially localized.

When  $F(\xi)=\frac{1}{2}|\xi|^2$, it is well known that equation (\ref{NLS0}) represents the classical critical nonlinear Schr\"odinger equation
\begin{equation}\label{1.3-0}
		i \partial_{t}u+\frac{1}{2}\Delta u=\lambda|u|^\frac{2}{n} u.
\end{equation}
For  $\lambda \in \mathbb{R}, n=1,2,3$, Hayashi-Naumkin\cite{HN03} proved the asymptotics for  the small solution  $u(t)$:  There exists  $W(x)\in L^2(\mathbb{R}^n)\cap L^\infty (\mathbb{R}^n)$ such that
\begin{equation}
	u(t,x)=\frac{1}{(it)^{n/2}}e^{i\frac{|x|^2}{4t}}W(\frac{x}{t})e^{i|W(\frac{x}{t})|^2\log t}+O_{L^\infty _x}(t^{-n/2-\beta}),\ \beta>0.\notag
\end{equation}
as  $t\rightarrow \infty $.
One note that this is not linear scattering, but rather a modified linear scattering.
The idea is to apply the operator  $\mathcal{F} U(-t)$ to the equation (\ref{1.3-0}) and use  the  factorization technique of the Schr\"odinger operator to derive the  ODE for  $\mathcal{F} U(-t)u(t)$:
\begin{equation}
	i\partial_{t}\mathcal{F} U(-t)u=\lambda t^{-1} |\mathcal{F} U(-t)u(t)|^{\frac{2}{n}}\mathcal{F} U(-t)u+R(t,\xi),\label{381}
\end{equation}
where   $U(t)=:e^{\frac{it}{2}\Delta }$ and  $\|R(t,\xi)\|_{L^\infty _\xi}=O(t^{-1-\beta})$ is integrable as  $t\rightarrow \infty $. By a standard ODE argument,  they deduce from (\ref{381}) the  asymptotic formula for  $\mathcal{F} U(-t)u(t)$ and then in  the solution  $u(t)$.


For the asymptotics of the classical critical nonlinear Schr\"odinger equation (\ref{1.3-0}) in one-dimension with  $\lambda \in \mathbb{R}$, there are four alternate approaches. In the paper \cite{DZ03}, Deift and Zhou established the asymptotics of the solution for large data in the defocusing case. The proof is  based on the complete integrability of (\ref{1.3-0}) and the inverse scattering techniques. In the paper \cite{LS2006},  Lindblad-Soffer made the change of variables  $u(t,x)=t^{-1/2} e^{i\frac{x^2}{4t}}v(t,\frac{x}{t})$ that allowed one to rewrite the equation (\ref{1.3-0}) as
\begin{equation}
	i\partial_{t}v=\lambda t^{-1}|v|^2v-t^{-2}\partial_{xx}v.\label{3101}
\end{equation}
For small and  smooth initial values, they proved that the remainder term  $t^{-2} \|\partial_{xx}v\|_{L^\infty _x}$ is integrable as  $t\rightarrow \infty $.  Then by standard ODE argument, they deduce from (\ref{3101})  the asymptotics of  $v$ and then in the solution  $u$. In the paper \cite{KP2011}, by performing an analysis in Fourier space, Kato-Pusater  derive the ODE (\ref{381})  through a very natural stationary phase argument. By standard ODE argument, they obtained the same result as \cite{HN03} for  $n=1$.  Their  proof is inspired by the space-time resonances method that introduced by Germain, Masmoudi and Shatah  \cite{Ger1,Ger2}.    In the paper \cite{Ifrim-Tataru}, Ifrim-Tataru  developed the wave packet test method: Measure the decay of the solution along the ray  $\Gamma_v =\left\{x=vt\right\} $, traveling with velocity  $v$, and denote
\begin{equation}
	\gamma(t,v)=\int _{\mathbb{R}} u(t,x) \overline{\Psi_v}(t,x)dx,\notag
\end{equation}
where  $\Psi_v(t,x)$ is the following wave packet test function
\begin{equation}
	\Psi_v(t,x)=\chi (\frac{x-vt}{\sqrt {t}})e^{i\frac{x^2}{2t}},\ \chi \in C_0^\infty (\mathbb{R})\ \text{and } \int _{\mathbb{R}}\chi (z)dz=1.\notag
\end{equation}
Since  $\Psi_v(t,x)$ is indeed an approximate solution of the linear system, with  $O(1/t)$ errors, Ifrim-Tataru \cite{Ifrim-Tataru}   established  the ODE for  $\gamma(t,v)$:
\begin{equation}
	\partial_{t}\gamma(t,v)=-it^{-1} |\gamma(t,v)|^2\gamma(t,v)+O_{L^\infty _v}(t^{-1/4+C\varepsilon ^2})\notag
\end{equation}
and obtained the asymptotic formula of  $\gamma(t,v)$. Moreover, by proving that
\begin{equation}
t^{-1/2}\gamma(t,v)-e^{-i\frac{v^2}{2}t}u(t,vt)=O_{L^2_v}(t^{-1+C\varepsilon ^2})\cap O_{L^\infty _v}(t^{-3/4+C\varepsilon ^2})\notag
\end{equation}
\begin{equation}
	\gamma(t,\xi)-e^{i\frac{t}{2}\xi^2}(\mathcal{F} u)(t,\xi)=O_{L^2_\xi}(t^{-1/2+C\varepsilon ^2})\cap O_{L^\infty _\xi}(t^{-1/4+C\varepsilon ^2})\notag
\end{equation}
 and  using the asymptotic formula of  $\gamma(t,v)$, they obtained  the asymptotic formulas of the solution both in physical space and frequency space.

In this paper, we consider the global existence and asymptotics of the solutions to the modified Schr\"odinger equation (\ref{NLS0}). Since  $F(\xi)$ has no explicit expression, the methods used in   \cite{DZ03,HN03,Ifrim-Tataru,KP2011,LS2006} can not be applied directly, and we have to develop a new approach. The second author of this paper first studied the asymptotic behavior of the solution to the modified one-dimensional Schr\"odinger equation. Assuming that  $F(\xi)$ satisfies certain elliptic assumption, Zhang \cite{Zhang} obtained the asymptotics of the solution for small initial data, combining the vector fields method and a semiclassical analysis method introduced by Delort \cite{Delort}.   In this paper, we consider the two-dimensional case and  assume that  $F(\xi)$ satisfies the following elliptic assumptions:
\begin{equation}\label{F1}
	F(\xi)=F_1(\xi_1)+F_2(\xi_2),\  \xi=(\xi_1,\xi_2)\rightarrow F(\xi)\in \mathbb{R}
\end{equation}
is a smooth  function defined on $\mathbb{R}^2$, satisfying
\begin{equation}
	F_k(\xi_k)\in C^\infty(\mathbb{R}),	\ 0<c_k\leq F_k''(\xi_k)\leq d_k, \ \forall  \ \xi_k\in \mathbb{R},\label{F}
\end{equation}
for some positive constants $c_k,d_k$, $k=1,2$. For example,  we can choose  smooth functions $ F_k(\xi_k)$, which have  an expansion
\begin{equation}
	F_k(\xi_k)=c^2_{k,\pm}\xi^2_k+c^1_{k,\pm}\xi+c^0_{k,\pm}+c^{-1}_{k,\pm}\xi^{-1} +c^{-2}_{k,\pm}\xi^{-2}+\cdots \notag	
\end{equation}
when  $\xi_k$ goes to  $\pm \infty $, where  $c^2_{k,\pm}>0,\ k=1,2.$  In particular, we can choose  $F(\xi)=|\xi|^2$. Therefore the asymptotic formulas established in this paper can be seen as a generalization of  \cite{DZ03,HN03,Ifrim-Tataru,KP2011,LS2006}.



To state our result precisely, we now give some notations. For  $\psi \in L^1(\mathbb{R}^2)$,  the Fourier transform of  $\psi$  is represented as  $\mathcal{F} \psi (\xi)=\hat \psi(\xi)=(2\pi)^{-1}\int_{ \mathbb{R}^2}\psi(x) e^{-ix\cdot\xi}dx$.    $[A,B]$ denotes the commutator  $AB-BA$ and  $<x>=:\sqrt {1+|x|^2}$. Different positive constants we denote by the same letter $C$. We write  $f\lesssim g$ when  $f\le Cg$.     $\mathcal {S} (\mathbb{R}^2) $ and  $H^s(\mathbb{R}^2)$ denote the usual  Schwartz and Sobolev space respectively.     $L^p=L^p(\mathbb{R}^2)$  denotes the usual Lebesgue space with the norm  $\|\phi\|_{L^p}=(\int _{\mathbb{R}^2}|\phi(x)|^p dx)^{1/p}$ if  $1\le p<\infty $ and  $\|\phi\|_{L^\infty }=\text{ess.sup }\left\{|\phi(x)|;x\in \mathbb{R}^2\right\} $. Moreover, we denote   $\|(\phi(x),\psi(x))\|_{L^2}=:(\int _{\mathbb{R}^2}|\phi(x)|^2 d x)^{1/2}+(\int _{\mathbb{R}^2}|\psi(x)|^2 d x)^{1/2}$.

 Throughout the paper, $F(\xi)=F_1(\xi_1)+F_2(\xi_2)$ denotes the  second order constant coefficients classical elliptic symbol, satisfying (\ref{F1})--(\ref{F}).  Since $F_k'$ is strictly increasing, there  are smooth strictly concave functions $\phi_k:\mathbb{R}\rightarrow \mathbb{R}$ such that  $	x_k+F'_k(d\phi_k(x_k))=0,\ k=1,2$, where we write  $d^j\phi_k(x_k)=\frac{d^j\phi_k(x_k)}{dx_k^j},\ j\ge1$ for simplicity.  Set  $\phi(x)=(\phi_1(x_1),\phi_2(x_2))$, then we have
\begin{equation}
x+F'(d\phi(x))=(x_1+F_1'(d\phi_1(x_1)),x_2+F_2'(d\phi_2(x_2)))=0,\label{221}
\end{equation}
which will be used frequently from now on.

We now state our results.
\begin{thm}\label{T1}
	Assume that  $\lambda \in \mathbb{C}$ with $\text{Im} \lambda \ge0$,  and the initial datum  $u_0\in H^2$,  $(x_k+F_k'(D))^2u_0\in L^2$, $k=1,2$, satisfying
	\begin{eqnarray}
		\|u_0\|_{H^2}+\sum_{k=1}^{2}\|(x_k+F_k'(D))^2u_0\|_{L^2} \le 1.\label{10291}
	\end{eqnarray}
	Then for $\varepsilon >0$ sufficiently small,    the Cauchy problem
	\begin{equation}
			\left\{\begin{array}{ll}
			(D_t -F(D))u=\lambda|u| u,&t>1,x\in\mathbb{R}^2
			\\ u(x,1)=\varepsilon u_0(x),
		\end{array}\label{NLS}
		\right.
	\end{equation}
 admits a unique global  solution $u\in C([1,\infty);L^2)\cap L_{\text{loc}}^6([1,+\infty);L^3)$,  satisfying  the pointwise estimate
	\begin{equation}
		\|u(t,\cdot)\|_{L^\infty} \lesssim  \varepsilon t^{-1 },\label{decay}
	\end{equation}
	for all $t>1$.
\end{thm}
\begin{rem}
	 The fact that the initial datum is given at time  $t=1$ does not have   deep meaning: it is simply more convenient when performing estimates,   since the  $L^\infty $ decay of  $\frac{1}{t}$
	 given by the linear part of the equation is not integrable at  $t=0$.
\end{rem}

In order to investigate the large time asymptotic behavior of the solution,  the following theorem plays an important role.
\begin{thm}\label{T2}
	Suppose that  $u_0$ satisfies (\ref{10291}),  $u$ is the solution obtained in Theorem \ref{T1}  and  $v_{\Lambda }$ is  the function defined by (\ref{827w1}) and (\ref{827w2}). Let
	\begin{equation}
		\Phi(t,x)=\int_1^t s^{-1} |v_{\Lambda }(s,x)| ds.\label{Phi}
	\end{equation}
	Then for $\varepsilon >0$ sufficiently small, there exists a unique complex function $z_+(x)\in L^2_x\cap L^\infty _x$ such that
	\begin{eqnarray}
	v_{\Lambda }(t,x)\exp \left(-i(w(x)t+\lambda \Phi(t,x))\right)-z_+(x)=O_{L^\infty _x}(t^{-1/2+C\varepsilon })\cap O_{L^2_x}(t^{-1+C\varepsilon }),\notag
	\end{eqnarray}
	holds as $t\rightarrow \infty $, where $w(x)=x\cdot d\phi(x)+F(d\phi(x))$.
\end{thm}
By Theorem \ref{T2}, we obtain the following asymptotics of the solution.
\begin{thm}\label{T3}
	Suppose that  $\text{Im} \lambda =0$,  $u_0$ satisfies (\ref{10291}) and  $u$ is the solution obtained in Theorem \ref{T1}.  Then for  $\varepsilon >0$ sufficiently small,  the followings hold:\\
	(a) Let  $z_+(x)$ be the asymptotic function obtained in Theorem \ref{T2} and
	\begin{equation}
		\phi_+(x)=\int _1^\infty s^{-1} (|v_{\Lambda }(s,x)|-|z_+(x)|)ds.\label{phi}
	\end{equation}
	The asymptotic formula
	\begin{equation}
		u(t,x)=\frac{1}{t} e^{i(tw(\frac{x}{t})+\lambda \phi_+(\frac{x}{t}) +\lambda |z_+(\frac{x}{t})|\log t  )}z_+(\frac{x}{t})+O_{L^\infty _x}(t^{-3/2+C\varepsilon })\cap O_{  L^2_x}(t^{-1+C\varepsilon }) \label{T3a}
	\end{equation}
	holds as $t\rightarrow \infty $. \\
	(b) The following  modified linear scattering formula holds:
	\begin{equation}
		\lim_{t\rightarrow \infty } \|u(t,x)-e^{i\lambda  (\phi_+(\frac{x}{t})+|z_+(\frac{x}{t})|\log t)}e^{iF(D)t}u_+\|_{L^2_x}=0. \label{T3b}
	\end{equation}
where
\begin{equation}
	u_+(x)=: -\frac{i}{2\pi} \int _{\mathbb{R}^2} e^{ix\cdot d\phi(y)}\frac{1}{\sqrt {\prod_{k=1}^2F_k''(\phi_k(y_k) ) }}z_+(y)dy.\label{274}
\end{equation}
\end{thm}

When  $\text{Im} \lambda >0$, (\ref{1.3-0}) represents the standard dissipative nonlinear Schr\"odinger equation. It modes the evolution of pulses propagating through optical fibers and in the field of optical fiber engineering(see e.g. \cite{Agrawal}). In the paper \cite{CPDE}, Shimomura first treated the dissipative nonlinear Schr\"odinger equation (\ref{1.3-0})  and established the asymptotic formula of the small solutions for  $n=1,2,3.$ He proved that the dissipative effects are visible not only in the phase correction of the asymptotic profile but also in the decay rate of the solution. In fact, the uniform norm of the solution decays like
\begin{equation}
	\|u(t,x)\|_{L^\infty _x}\lesssim (t\log t)^{-n/2},\notag
\end{equation}
which decays faster than the free solution.  After the work of Shimomura \cite{CPDE}, there are many papers devoted to studying the decay estimates and the asymptotics of the solutions to the  dissipative nonlinear Schr\"odinger equation, see e.g. \cite{HLN1,HLN2,Jin,Kita,Kita2,Xuan,Ogawa,Sato} and references therein.

Inspired  by the work of Shimomura \cite{CPDE}, we also study the influence of the dissipative effects on the  decay estimates and the asymptotics of the solutions to the modified two-dimensional Schr\"odinger equation. Our results are the following.
\begin{thm}
	\label{T4}
Suppose that  $\text{Im} \lambda >0$,  $u_0$ satisfies (\ref{10291}) and  $u$ is the solution obtained in Theorem \ref{T1}.  Then for  $\varepsilon >0$ sufficiently small,  the followings hold:\\
	(a) Let  $z_+(x)$ be the asymptotic function obtained in Theorem \ref{T2} and
	\begin{equation}
		\psi_+(x)=\text{Im} \lambda  \int_1^\infty s^{-1} \left(|v_{\Lambda }(s,x)|e^{\text{Im} \lambda \Phi(s,x)}-|z_+(x)|\right)ds,\label{psi}
	\end{equation}
	\begin{equation}
		S(t,x)=\frac{1}{\text{Im} \lambda }\log \left(1+\text{Im} \lambda |z_+(x)|\log t+\psi_+(x)\right).\label{S}
	\end{equation}
	The asymptotic formula
	\begin{equation}
		u(t,x)=\frac{1}{t} e^{i(tw(\frac{x}{t})+\lambda S(t,\frac{x}{t}))}z_+(\frac{x}{t})+O_{L^\infty _x}(t^{-3/2+C\varepsilon })\cap O_{  L^2_x}(t^{-1+C\varepsilon })\label{T4a}
	\end{equation}
	holds as $t\rightarrow \infty $.\\
	(b) The following  modified linear scattering formula holds:
	\begin{equation}
		\lim_{t\rightarrow \infty } \|u(t,x)-e^{i\lambda S(t,\frac{x}{t})}e^{iF(D)t}{u}_+(x)\|_{L^2_x}=0, \label{T4b}
	\end{equation}
where  $u_+$ is the function defined in (\ref{274}).\\
	(c) If $u_0\neq 0$, then the following limit exists
	\begin{equation}
		\lim_{t\rightarrow \infty } (t\log t)\|u(t,x)\|_{L^\infty _x} =\frac{1}{\text{Im} \lambda }.\label{T4c}
	\end{equation}
\end{thm}
\begin{rem}
	Note that 	the limit (\ref{T4c})  is  independent of the initial value  $u_0$.  This property was first established  in  \cite{CaJFA}  for a class of nonvanishing initial data.
\end{rem}
\begin{rem}
	By the definition of $S(t,x)$, we can write the modification factor $e^{i\lambda S(t,x)}$ explicitly:
	\begin{equation}
		e^{i\lambda S(t,x)}=\frac{\exp \left({\frac{i\text{Re}\lambda }{\text{Im} \lambda }\log (1+\text{Im} \lambda |z_+(x)|\log t+\psi_+(x))}\right)}{1+\text{Im} \lambda |z_+(x)|\log t+\psi_+(x)}.\label{T4d}
	\end{equation}
\end{rem}

 We  briefly sketch the strategy used to derive our main results.  We adopt the  semiclassical analysis method  introduced  by Delort  \cite{Delort}, see also \cite{Xuan,S,Zhang} which are closer to the problem we are considering. We make first a semiclassical change of variables
\begin{equation}
	u(t,x)=\frac{1}{{t}}v(t,\frac{x}{t}) \label{827w1}
\end{equation}
for some new unknown function $v$,  that allows to  rewrite   the equation  (\ref{NLS})    as
\begin{equation}\label{1.15}
	(D_t-G_h^w(x\cdot\xi+F(\xi)))v=\lambda h|v|v,
\end{equation}
where the semiclassical parameter $h=\frac1t$, and the Weyl quantization of a symbol $a$ is given by
\begin{equation}
	G_h^w(a)u(x)=\frac{1}{(2\pi h)^2}\int _{\mathbb{R}^2}\int_{ \mathbb{R}^2}e^{\frac{i}{h}(x-y)\cdot \xi}a(\frac{x+y}{2},\xi)u(y)dyd\xi.\notag
\end{equation}
We remark that the operator
\begin{equation}
	\mathcal{L}=(\mathcal{L} _1, \mathcal{L} _2)\qquad  \text{with  } \qquad\mathcal{L} _k=\frac{1}{h} G_h^w(x_k+F_k'(\xi_k)),\qquad k=1,2.\label{L}
\end{equation}
commutes exactly to the linear part of the equation (\ref{1.15}) and  a Leibniz rule holds for the action of  $\mathcal{L} $ on  $|v|v$.   Actually, by (\ref{F}) and (\ref{221}), the quotient  $e_k(x,\xi)=\frac{x_k+F_k'(\xi_k)}{\xi_k-d\varphi_k(x_k)}$ is smooth and  $|e_k(x,\xi)|$ stays between two positive constants.
Consequently, one may write, using symbolic calculus for semiclassical operators
\begin{equation}
	\mathcal{L}_k =\frac{1}{h}G_h^w(e_k(\xi_k-d\phi_k(x_k)))=G_h^w(e_k)\left[\frac{1}{h}G_h^w(\xi_k-d\phi_k(x_k))\right]+G_h^w(r_k), k=1,2\notag
\end{equation}
with some other symbol  $r_k$. If one makes act the main contribution on the right-hand side of the above equality on  $|v|v$, one gets ($D_k=:\frac{\partial_{x_k}}{i}$)
\begin{eqnarray*}	
&&G_h^w(e_k)\left[\left(D_k-\frac{1}{h}d\phi_k(x_k)\right)(|v|v)\right]\\
&=&G_h^w(e_k)\left[\frac{3}{2}|v|\left(D_k-\frac{1}{h}d\phi_k(x_k)\right)v-\frac{1}{2}|v|^{-1}v^2\overline{\left(D_k-\frac{1}{h}d\phi_k(x_k)\right)v}\right]\notag
\end{eqnarray*}
that is a quantity whose  $L^2(\mathbb{R}^2)$ norm may be bounded by   $C\|v\|_{L^\infty }\|\mathcal{L} v\|_{L^2}$ (if one re-expresses  on the right-hand side  $\left(D_k-\frac{1}{h}d\phi_k(x_k)\right)v$ from  $\mathcal{L}_k v$).   In other words,  $\mathcal{L} $  obeys a Leibniz rule when acting on	 $|v|v$.
Repeating  this procedure, one sees that the nonlinearity  $|v|v$ not only obey a Leibnitz rule for  $\mathcal{L} $ but also for the operator  $\mathcal{L} ^2=:(\mathcal{L} _1^2,\ \mathcal{L} ^2_2)$(see Lemma \ref{l2v})
\begin{equation}
	\|\mathcal{L} ^2(|v|v)\|_{L^2_x}\lesssim  \|v\|_{L^\infty } (\|v\|_{L^2_x}+\|\mathcal{L} ^2v\|_{L^2_x}).\label{826w1}
\end{equation}
Applying the operator $\mathcal{L} ^2$ to  (\ref{1.15}), then using  (\ref{826w1}), one  obtains an energy inequality of the form
\begin{eqnarray}
	&&\|\mathcal{L} ^2v(t,\cdot)\|_{L^2}\label{826w2}\\
&\lesssim& \sum_{k=1}^{2}\|(x_k+F'_k(D))^2u_0\|_{L^2}+\int_1^t \|v(\tau,\cdot)\|_{L^\infty } (\|v(\tau,\cdot)\|_{L^2}+\|\mathcal{L} ^2v(\tau,\cdot)\|_{L^2})\frac{d\tau}{\tau}.\nonumber
\end{eqnarray}
If one has an a priori estimate  $\|v(\tau,\cdot)\|_{L^\infty }=O(\varepsilon )$, Gronwall lemma provides for the left-hand side of (\ref{826w2}) a  $O(t^{C\varepsilon })$ bound.

 On the other hand, one can establish from an a priori   $\|\mathcal{L} ^2v(t,\cdot)\|_{L^2}=O(t^{C\varepsilon })$ an  $L^\infty $  estimate for  $v$,  by deducing
from (\ref{1.15}) an ODE satisfied by  $v$. Actually, if one develops the symbol  $x\cdot \xi+F(\xi)$  at  $\xi=d\phi(x)$, one gets, using  that  $\nabla _{\xi}\left(x\cdot \xi+F(\xi)\right)|_{\xi=d\phi(x)}=0$ (see (\ref{221})),
\begin{equation}
	x\cdot \xi+F(\xi)=w(x)+\sum_{k=1}^{2}\int_{0}^{1} F_k''(\theta\xi_k+(1-\theta)d\phi_k(x_k)) (1-\theta)\mathrm{d}\theta (\xi_k-d\phi_k(x_k))^2,\label{15100}
\end{equation}
where  $w(x)=x \cdot d\phi(x)+F(d\phi(x))$. Taking the Weyl quantization and using the symbol
calculus, one deduces from (\ref{1.15}) an ODE for  $v$
\begin{equation}
	D_tv=w(x)v+\lambda h|v| v+h^2\sum_{k=1}^{2}G_h^w(b_k)\circ \mathcal{L} ^2_k,\label{2261}
\end{equation}
where  $b_k, \ k=1,2$ are some  other symbols. Assume for a while that one can deduce from the a priori estimate  $\|\mathcal{L} ^2v(t,\cdot)\|_{L^2}=O(t^{C\varepsilon })$ that  $\|h^2\sum_{k=1}^{2}G_h^w(b_k)\mathcal{L} ^2_kv\|_{L^\infty _x}$ is  time integrable, one derives	  from the ODE (\ref{2261}) a uniform  $L^\infty $ control of  $v$.  Putting together these  $L^2$ and  $L^\infty $ estimates and performing a bootstrap argument, one finally shows that (\ref{1.15}) has global solutions and determines their asymptotic behavior.

To estimate   $\|h^2\sum_{k=1}^{2}G_h^w(e_k)\mathcal{L} ^2_kv\|_{L^\infty _x}$ one would be tempted to use the  semiclassical Sobolev inequality. Note however that  this would lead to an energy norm of  $v$  containing  the spatial derivative of order three, and we can not recover   this norm from  the  equation (\ref{1.15}) via  the standard energy method due to the lack of the regularity of  $|v|v$.  We instead use the operators whose symbols are localized  in a neighbourhood of $M:=\{(x,\xi)\in\mathbb{R}^2\times \mathbb{R}^2: x+F'(\xi)=0\}$ of size $O(\sqrt h)$. In that way, we can apply Proposition 2.3 to pass uniform norms of the remainders to the  $L^2$ norm.
More precisely, we set
\begin{eqnarray}
	v_{\Lambda }=G_h^w(\gamma(\frac{x+F'(\xi)}{\sqrt h}))v, \label{827w2}
\end{eqnarray}
where $\gamma\in C_0^\infty (\mathbb{R}^2)$ satisfying $\gamma=1$ in a neighbourhood of zero.
Applying  $G_h^w(\gamma(\frac{x+F'(\xi)}{\sqrt{h}}))$ to the equation (\ref{1.15}), one obtains  the ODE for $v_{\Lambda }$
\begin{equation}
	D_t v_\Lambda=w(x)v_{\Lambda }+\lambda h  |v_\Lambda| v_\Lambda+ R(v), \label{826w3}
\end{equation}
where the remainder
\begin{eqnarray}
	R(v)&=&[D_t-G^w_h(x\cdot\xi+F(\xi)),G^w_h(\gamma(\frac{x+F'(\xi)}{\sqrt{h}}))]v+G_h^w(x\cdot\xi+F(\xi)-w(x))v_{\Lambda}\notag \\
	&&- \lambda hG^w_h(1-\gamma(\frac{x+F'(\xi)}{\sqrt{h}}))(|v| v)+\lambda h\left( |v| v-|v_\Lambda| v_\Lambda\right).\notag
\end{eqnarray}
  To derive from (\ref{826w3}) and the a priori estimate of the form  $\|\mathcal{L} ^2v(t,\cdot)\|_{L^2_x}=O(t^{C\varepsilon })$ a uniform  $L^\infty $ control of  $v$, we need to show that  $R(v)$ is a  time integrable error term and  $v_{\Lambda ^c}=:v-v_{\Lambda }$ decays faster than the main part  $v_{\Lambda }$.   The proof of these  two estimates will  strongly use the semiclassical pseudo-differential calculus and constitutes the most technical part of the paper.

The framework of this paper is organized as follows. In Section \ref{s2}, we will present the  definitions and properties of Semiclassical pseudo-differential operators. In Section \ref{s3},  we use the Strichartz estimate  and the semiclassical analysis method to prove Theorem \ref{T1}.  In Section \ref{s4}, we prove Theorems \ref{T2}--\ref{T4} by deducing the
long-time behavior of solutions from the associated ODE dynamics. Finally, in the Appendix, we give the proof of Lemma \ref{l2v}.

\section{Semiclassical pseudo-differential operators}\label{s2}
In order to prove an  $L^\infty $ estimate on  $u$ we need to reformulate the starting problem (\ref{NLS}) in terms of an
ODE satisfied by a new function  $v$ obtained from  $u$, and this will strongly use
the semiclassical pseudo-differential calculus. In this section,
we introduce this semiclassical environment, defining classes of symbols and
operators we shall use and several useful properties.  A general reference is Chapter 7 of the book of Dimassi-Sj\"{o}strand \cite{D-S} or Chapter 4 of the book of Zworski \cite{Zworski}.
\begin{defn}
	An order function on $\mathbb{R}^2\times\mathbb{R}^2$ is a smooth map from $\mathbb{R}^2\times\mathbb{R}^2$ to $\mathbb{R}_+$: $(x,\xi)\rightarrow M(x,\xi)$ such that there are $N_0\in\mathbb{N}$ and $C>0$,
	\begin{equation}
		M(y,\eta)\leq C<x-y>^{N_0} <\xi-\eta>^{N_0}M(x,\xi),\notag
	\end{equation}
	holds for any  $(x,\xi)$, $(y,\eta)\in \mathbb{R}^2\times\mathbb{R}^2$.
\end{defn}
\begin{defn}\label{26w2}
	Let   $M$ be an order function on $\mathbb{R}^2\times\mathbb{R}^2$, $\delta\geq0$.  One denotes by $S_\delta(M)$ the space of smooth functions
	$\mathbb{R}^2\times\mathbb{R}^2\times(0,1]\rightarrow \mathbb{C}$
	satisfying for any $\alpha,\beta\in\mathbb{N}^2$
	\begin{equation}\label{2.3}
		|\partial_x^{\alpha}\partial_\xi^\beta a(x,\xi,h)|\leq C_{\alpha,\beta} M(x,\xi)h^{-(|\alpha|+|\beta|)\delta}.
	\end{equation}
\end{defn}
Recall that  $d\phi_k(x_k),k=1,2$ are the functions defined in (\ref{221}). Let
\begin{equation}
	e_k(x,\xi)=\frac{x_k+F_k'(\xi_k)}{\xi_k-d\phi_k(x_k)},\qquad \widetilde{e}_k(x,\xi)=\frac{\xi_k-d\phi_k(x_k)}{x_k+F_k'(\xi_k)},\ k=1,2. \label{7131}
\end{equation}	
By Definition \ref{26w2}, we have the following lemma.
\begin{lem}\label{lone}
	Assume that  $e_k(x,\xi),\ \widetilde{e}_k(x,\xi), d^2\phi_k(x_k),k=1,2$ are the functions defined in (\ref{221}) and (\ref{7131}).
	Then we have $	e_k(x,\xi),\ \widetilde{e}_k(x,\xi), \ d^2\phi_k(x_k) \in S_{0}(1),\ k=1,2.$
\end{lem}
\begin{proof}
	Taking a derivative of $
	x_k+F_k'(d\phi_k(x_k))=0
	$,
	one gets by using (\ref{F})
	\begin{equation}
		|d^2\phi_k(x_k)|=|-\frac{1}{F_k''(d\phi_k(x_k))}|\lesssim 1.\notag
	\end{equation}
	Moreover, by induction and (\ref{F}), one gets  $|d^{j}\phi_k(x_k)|\lesssim 1,\ \forall \ j\ge2$, i.e.  $d^2\phi_k(x_k)\in S_0(1)$.
	On the other hand,   since $x_k+F'_k(d\phi_k(x_k)=0$, one  can rewrite $e_k$ as following
	\begin{equation}
		e_k(x,\xi)=\frac{F_k'(\xi_k)-F_k'(d\phi_k(x_k))}{\xi_k-d\phi_k(x_k)}=\int^1_0 F_k''(t\xi_k+(1-t)d\phi_k(x_k))dt,\ k=1,2.\label{2121}
	\end{equation}
	 Using  $d^2\phi_k(x_k)\in S_0(1)$ and   (\ref{2121})
	we can easily check that $e_k\in S_0(1)$, $k=1,2$.  Similarly, we  can get $\widetilde{e}_k\in S_0(1)$, $k=1,2$ and omit the details.
\end{proof}

\begin{defn}
	Let   $M$ be an order function on $\mathbb{R}^2\times\mathbb{R}^2$, $\delta\geq0$. For any symbol  $a\in S_\delta (M)$, we define the Weyl quantization  $G^w_h(a)$ to be the operator
	\begin{equation}
		G^w_h(a)u(x)=\frac{1}{(2\pi h)^2}\int_{\mathbb{R}^2}\int_{\mathbb{R}^2}
		e^{\frac{i}{h}(x-y)\cdot\xi}a(\frac{x+y}{2},\xi)u(y)dyd\xi,\qquad \forall u\in \mathcal{S}(\mathbb{R}^2).\label{2.4}
	\end{equation}
\end{defn}

We have the following basic  properties for the Weyl quantization.
\begin{prop}[Theorem 4.1 in \cite{Zworski}]\label{pz}
	 Let   $M$ be an order function on $\mathbb{R}^2\times\mathbb{R}^2$, $\delta\geq0$.
	Assume that   $a\in S_\delta (M)$ is a real-valued symbol. Then the  Weyl quantization $G_h^w(a)$ is self-adjoint on  $L^2(\mathbb{R}^n)$.
\end{prop}
\begin{prop}\label{pcomposition}(Theorem 7.3 in \cite{D-S})
	Let   $M$ be an order function on $\mathbb{R}^2\times\mathbb{R}^2$, $\delta\geq0$. For any $a,b\in S_\delta(M)$
	\begin{eqnarray}
		G_h^w(a\sharp b)=G_h^w(a)\circ G_h^w(b)\notag
	\end{eqnarray}
	where the Moyal product of the symbols is defined by
	\begin{equation}
		a\sharp b(x,\xi):=\frac{1}{(\pi h)^{4}}\int_{\mathbb{R}^2}\int_{\mathbb{R}^2}\int_{\mathbb{R}^2}\int_{\mathbb{R}^2}
		e^{\frac{2i}{h}(\eta \cdot z- y\cdot \zeta)}a(x+z,\xi+\zeta)b(x+y,\xi+\eta)dyd\eta dz d\zeta.\label{pcomption}
	\end{equation}
\end{prop}
It is often useful to derive an asymptotic expansion for $a\sharp b$, as it allows easier computations than the integral formula  (\ref{pcomption}). For the one-dimensional case, we refer to the Appendix of \cite{S}.
\begin{prop}\label{premainder}
	Let   $M$ be an order function on $\mathbb{R}^2\times\mathbb{R}^2$, $\delta\geq0$. For any $a,b\in S_\delta(M)$
	\begin{equation}
		a\sharp b(x,\xi) =ab (x,\xi)+\frac{ih}{2} \sum_{j=1}^{2}\left|\begin{array}{c c}
			\partial_{x_j}a (x,\xi)	&\partial_{\xi_j}a (x,\xi)  \\
			\partial_{x_j}b(x,\xi) &\partial_{\xi_j}b   (x,\xi)
		\end{array}\right|
		+ r^{a\sharp b},\notag
	\end{equation}
	where
	\begin{eqnarray}
		r^{a\sharp b}
		&=&-\frac{h^2}{2}\sum_{\substack{	\alpha =(\alpha _1,\alpha _2)\\
				|\alpha |=2}} \frac{(-1)^{|\alpha _1|}}{\alpha !} \frac{1}{(\pi h)^4} \int _{\mathbb{R}^2} \int_{\mathbb{R}^2}\int _{\mathbb{R}^2} \int _{\mathbb{R}^2}  e^{\frac{2i}{h}(\eta\cdot z-y\cdot \zeta)}\int _0^1\partial_x ^{\alpha _1}\partial_\xi ^{\alpha _2}a(x+tz,\xi+t\zeta) \notag\\
		&& \qquad \times (1-t)dt \partial_x ^{\alpha _2}\partial_\xi ^{\alpha _1}b(x+y,\xi+\eta) dyd\eta dzd\zeta. \label{1}
	\end{eqnarray}
\end{prop}
\begin{proof}
	From (\ref{pcomption}) and  Taylor's  formula
	\begin{eqnarray}
		a(x+z,\xi+\zeta)&=&a(x,\xi)+\sum_{\substack{	\alpha =(\alpha _1,\alpha _2)\\
				|\alpha |=1}} \partial_x ^{\alpha _1}\partial_\xi ^{\alpha _2} a(x,\xi) z^{\alpha _1}\zeta^{\alpha _2} \notag\\
		&&	+\sum_{\substack{	\alpha =(\alpha _1,\alpha _2)\\
				|\alpha |=2}}\frac{2}{\alpha !} \int_0^1 \partial_x ^{\alpha _1}\partial_\xi ^{\alpha _2} a(x+tz,\xi+t\zeta)z^{\alpha _1}\zeta^{\alpha _2}(1-t)dt,\notag
	\end{eqnarray}
	we have
	\begin{eqnarray}
		&&a\sharp b (x,\xi)\notag\\
		&=&\frac{1}{(\pi h)^4} \int _{\mathbb{R}^2} \int_{\mathbb{R}^2}\int _{\mathbb{R}^2} \int _{\mathbb{R}^2}  e^{\frac{2i}{h}(\eta\cdot z-y\cdot \zeta)}a(x,\xi)b(x+y,\xi+\eta) dyd\eta dzd\zeta \notag\\
		&&+\frac{1}{(\pi h)^4} \int _{\mathbb{R}^2} \int_{\mathbb{R}^2}\int _{\mathbb{R}^2} \int _{\mathbb{R}^2}  e^{\frac{2i}{h}(\eta\cdot z-y\cdot \zeta)} \sum_{\substack{	\alpha =(\alpha _1,\alpha _2)\\
				|\alpha |=1}} \partial_x ^{\alpha _1}\partial_\xi ^{\alpha _2} a(x,\xi) b(x+y,\xi+\eta)z^{\alpha _1}\zeta^{\alpha _2}dyd\eta dzd\zeta \notag\\
		&& +\frac{1}{(\pi h)^4} \int _{\mathbb{R}^2} \int_{\mathbb{R}^2}\int _{\mathbb{R}^2} \int _{\mathbb{R}^2}  e^{\frac{2i}{h}(\eta\cdot z-y\cdot \zeta)}\sum_{\substack{	\alpha =(\alpha _1,\alpha _2)\\
				|\alpha |=2}}\frac{2}{\alpha !} \int_0^1 \partial_x ^{\alpha _1}\partial_\xi ^{\alpha _2} a(x+tz,\xi+t\zeta)z^{\alpha _1}\zeta^{\alpha _2}\notag \\
		&& \qquad \times (1-t)dt b(x+y,\xi+\eta) dyd\eta dzd\zeta\notag\\
		&=:& I_1+I_2+I_3. \notag
	\end{eqnarray}
	By direct computation, we get
	\begin{equation}
		I_1=a (x,\xi) \int _{\mathbb{R}^2} \int _{\mathbb{R}^2} b(x+y,\xi+\eta) \delta_0(y) \delta_0 (\eta) dyd\eta =  a(x,\xi) b(x,\xi),\notag
	\end{equation}
	where  $\delta_0$ is the Dirac function.  	Integrate  $I_2$ by parts via the identity  $$z^{\alpha _1}\zeta^{\alpha _2} e^{\frac{2i}{h}(\eta\cdot z-y\cdot \zeta)}=\frac{h}{2i} (-1)^{|\alpha _2|} \partial_{\eta}^{\alpha _1}\partial_y^{\alpha _2} e^{\frac{2i}{h}(\eta\cdot z-y\cdot \zeta)},\qquad |\alpha _1|+|\alpha _2|=1$$  to get
	\begin{eqnarray}
		I_2          	&=& \frac{1}{(\pi h)^4} \sum_{\substack{	\alpha =(\alpha _1,\alpha _2)\\
				|\alpha |=1}}\int _{\mathbb{R}^2} \int_{\mathbb{R}^2}\int _{\mathbb{R}^2} \int _{\mathbb{R}^2}\frac{h}{2i} (-1)^{|\alpha _1|}   e^{\frac{2i}{h}(\eta\cdot z-y\cdot \zeta)}\partial_x ^{\alpha _1}\partial_\xi ^{\alpha _2} a(x,\xi) \notag\\
		&&\qquad \times \partial_y^{\alpha _2} \partial_{\eta} ^{\alpha _1} b(x+y,\xi+\eta)  dyd\eta dzd\zeta\notag\\
		&=& \frac{h}{2i} \sum_{\substack{	\alpha =(\alpha _1,\alpha _2)\\
				|\alpha |=1}}(-1)^{|\alpha _1|} \partial_x ^{\alpha _1}\partial_\xi ^{\alpha _2}a(x,\xi)\partial_x ^{\alpha _2}\partial_\xi ^{\alpha _1}b(x,\xi)=\frac{ih}{2} \sum_{j=1}^{2}\left|\begin{array}{c c}
			\partial_{x_j}a 	&\partial_{\xi_j}a   \\
			\partial_{x_j}b &\partial_{\xi_j}b
		\end{array}\right|.\notag
	\end{eqnarray}
	The same calculation shows that $I_3$ is given by the right-hand side of (\ref{1}).     This completes the proof of Proposition \ref{premainder}.
\end{proof}
In particular, from Proposition \ref{premainder}, we have that
\begin{equation}
	(x_k+F_k'(\xi_k))^2=	(x_k+F_k'(\xi_k))\sharp 	(x_k+F_k'(\xi_k)),\qquad k=1,2.\notag
\end{equation}
Taking the Weyl quantization, then using Proposition  \ref{pcomposition}, we see that
\begin{equation}
	\mathcal{L} ^2_k=\frac{1}{h^2}G_h^w((x_k+F_k'(\xi_k))^2),\qquad k=1,2,\
\end{equation}
which will be used  frequently  in the rest of the paper without reference.

We will use the following boundedness of the Weyl quantization.
\begin{prop}\label{pjl2}(Theorem 7.11 in \cite{D-S})
	Let $a\in S_{\delta}(1)$, $\delta\in [0,\frac{1}{2}]$, then
	\begin{equation}
		\|G^w_h(a)\|_{\mathcal{L}(L^2,L^2)}\leq C,\ \textrm{ for all }h\in(0,1].\notag
	\end{equation}
\end{prop}
A key role in this paper will be played by symbols $a$ verifying (\ref{2.3}) with $M(x, \xi) = <\frac{x+F'(\xi)}{\sqrt{h}}>^{-N}$,
for $N \in\mathbb{N}$. Although $M(x,\xi)$ is no longer an order function
because of the term $h^{\frac{1}{2}}$, it is useful to  keep  the notation $a\in S_\delta(M)$ whenever $a,M$ verify (\ref{2.3}).
\begin{prop}
	\label{pj1}
	Let $a\in S_{\delta}(<\frac{x+F'(\xi)}{\sqrt{h}}>^{-2})$, $\delta\in [0,\frac{1}{2}]$, then
	\begin{equation}
		\|G^w_h(a)\|_{\mathcal{L}(L^2,L^\infty)}\lesssim  h^{-\frac{1}{2}},\ \textrm{ for all }h\in(0,1].\notag
	\end{equation}
\end{prop}
\begin{proof}
	By the definition of the operator $G^w_h(a)$ in (\ref{2.4}), we have
	\begin{eqnarray}
	G^w_h(a) f(x)
	&=& \frac{1}{(2\pi)^{2}  }\int_{\mathbb{R}^2}\int_{\mathbb{R}^2}
	e^{i(\frac{x}{\sqrt{h}}-y)\cdot \xi} a(\frac{x+\sqrt{h}y}{2},\sqrt{h}\xi)
	f(\sqrt{h}y)dyd\xi\nonumber\\
	&=& \frac{1}{(2\pi)^{4}  }\int_{\mathbb{R}^2}h^{-1}\hat{f}(\frac{\eta}{\sqrt{h}})d\eta
	\int_{\mathbb{R}^2}\int_{\mathbb{R}^2}
	e^{i(\frac{x}{\sqrt{h}}-y)\cdot \xi+i\eta\cdot  y} a(\frac{x+\sqrt{h}y}{2},\sqrt{h}\xi)
	dyd\xi.\label{24x2}
\end{eqnarray}
Since  $a\in S_{\delta}(<\frac{x+F'(\xi)}{\sqrt{h}}>^{-2})$, we have that for  $\forall \ \alpha ,\beta \in \mathbb{N}^2$
\begin{eqnarray}
	|\partial_x ^\alpha \partial_{\xi}^\beta a(\frac{x+\sqrt {h}y}{2},\sqrt {h}\xi)|&\lesssim& h^{(\frac{1}{2}-\delta)(|\alpha |+|\beta|)} 	<  \frac{\frac{x+\sqrt{h}y}{2}+F'(\sqrt{h}\xi)}{\sqrt{h}}>^{-2}\nonumber\\
&\lesssim& 	<  \frac{\frac{x+\sqrt{h}y}{2}+F'(\sqrt{h}\xi)}{\sqrt{h}}>^{-2}.\label{24x1}
\end{eqnarray}
	Integrating  (\ref{24x2}) by parts via the identity
		\begin{equation}
		\left(\frac{1-i(\frac{x}{\sqrt{h}}-y)\cdot \partial_\xi}{1+|\frac{x}{\sqrt{h}}-y|^2}
		\right)^{3}\left(\frac{1+i(\xi-\eta)\cdot \partial_y}{1+|\xi-\eta|^2}
		\right)^{3} e^{i(\frac{x}{\sqrt{h}}-y)\cdot \xi+i\eta\cdot  y}= e^{i(\frac{x}{\sqrt{h}}-y)\cdot \xi+i\eta\cdot  y}\notag
	\end{equation}
and using (\ref{24x1}), we get
\begin{eqnarray}
	|G^w_h(a) f(x)|
	&\lesssim & h^{-1}\int_{\mathbb{R}^2} \left|\hat{f}(\frac{\eta}{\sqrt{h}})
	\right|d\eta
	\int_{\mathbb{R}^2}\int_{\mathbb{R}^2}
	<\frac{x}{\sqrt{h}}-y>^{-3}<\xi-\eta>^{-3} \notag\\
	&&\times
	<  \frac{\frac{x+\sqrt{h}y}{2}+F'(\sqrt{h}\xi)}{\sqrt{h}}>^{-2}
	dyd\xi.\label{10142}
\end{eqnarray}
On the other hand,  by Young's inequality, we have
\begin{eqnarray}
	&&\|\int <\xi-\eta>^{-3}	<  \frac{\frac{x+\sqrt{h}y}{2}+F'(\sqrt{h}\xi)}{\sqrt{h}}>^{-2}d\xi\|_{L^2_\eta}\notag\\
	&\lesssim &  \|<\eta>^{-3}\|_{L^1_\eta}\|	<  \frac{\frac{x+\sqrt{h}y}{2}+F'(\sqrt{h}\xi)}{\sqrt{h}}>^{-2}\|_{L^2_\xi}\lesssim 1,\label{10141}
\end{eqnarray}
	where the last inequality holds since by  (\ref{F}) and (\ref{221})
	\begin{eqnarray*}
		\|<  \frac{\frac{x+\sqrt{h}y}{2}+F'(\sqrt{h}\xi)}{\sqrt{h}}>^{-2}\|_{L^2_\xi}
		&=&\|<\frac{F'(\sqrt {h}\xi)-F'(d\phi(\frac{x+\sqrt {h}y}{2}))}{\sqrt {h}}>^{-2}\|_{L^2_\xi} \notag\\
		&\lesssim &  \|<\xi-\frac{d\phi(\frac{x+\sqrt {h}y}{2})}{\sqrt {h}}>^{-2}\|_{L^2_\xi}\lesssim 1.
	\end{eqnarray*}
Applying Cauchy-Schwartz's inequality to (\ref{10142}) and using (\ref{10141}), we obtain
\begin{equation}
	|G^w_h(a) f(x)|
	\lesssim   h^{-1}
	\left\|\hat{f}(\frac{\eta}{\sqrt{h}})
	\right\|_{L^2_\eta} \int _{\mathbb{R}^2}<\frac{x}{\sqrt {h}}-y>^{-3}dy\lesssim h^{-1/2}\|f\|_{L^2}.\notag
\end{equation}
This completes the proof of Proposition \ref{pj1}.
\end{proof}
\begin{prop}\label{pj2}
	Let  $h\in(0,1]$ and $a(\xi)$ be a smooth function that   satisfies $|\partial_\xi^\alpha a(\xi)|\leq C_\alpha <\xi>^{-2-|\alpha|}$ for any $\alpha \in \mathbb{N}^2$.  Then
	\begin{equation}
		\|G^w_h(a(\frac{x+F'(\xi)}{\sqrt{h}}))\|_{\mathcal{L}(L^2,L^\infty)}=O(h^{-\frac{1}{2}}),
		\ \|G^w_h(a(\frac{x+F'(\xi)}{\sqrt{h}}))\|_{\mathcal{L}(L^2,L^2)}=O(1).\notag
	\end{equation}
\end{prop}
\begin{proof}
	Since $a(\frac{x+F'(\xi)}{\sqrt{h}})\in S_{\frac{1}{2}}(<\frac{x+F'(\xi)}{\sqrt{h}}>^{-2})\cap S_{\frac{1}{2}}(1)$, from Propositions \ref{pjl2}--\ref{pj1}, we have the estimates in  Proposition \ref{pj2}.
\end{proof}

The rest of this section is devoted to  establishing some technical results in the semiclassical framework. More precisely, Lemmas \ref{l1}--\ref{l3}   will be often recalled to compute the composition of the symbols. While  Lemmas \ref{l4}--\ref{sharp2} deal with the boundedness of the operators and the commutators of the symbols,  which will be used to prove Lemmas \ref{l4.3}--\ref{l4.4}.

\begin{lem}\label{l1}
	Suppose that  $a_{k}(x,\xi)\in S_0(1),\ 1\le k\le2$. There exist symbols $b_{k}(x,\xi)$, $c_{k}(x,\xi)$, $d_{k}(x,\xi)\in S_0(1), \ 1\le k\le2$ such that
\begin{equation}
		a_{k}(x,\xi)(x_k+F_k'(\xi_k))=a_{k}(x,\xi)\sharp (x_k+F_k'(\xi_k))+hb_{k}(x,\xi),\label{l1.1}
\end{equation}
and
\begin{equation}
	a_{k} (x,\xi) \left(x_k+F_k'(\xi_k)\right)^2 =a_{k}(x,\xi)\sharp \left(x_k+F_k'(\xi_k)\right)^2 +h c_{k}(x,\xi)\sharp (x_k+F_k'(\xi_k))+h^2d_{k}(x,\xi), \label{l1.3}
\end{equation}
	for $k=1,2$. As a corollary, we have the following estimates for the commutators
	\begin{equation}
		\|[\mathcal{L} _k,G_h^w(a_k(x,\xi))]v\|_{L_x^2}\lesssim \|v\|_{L_x^2},\ \|[\mathcal{L} ^2_k,G_h^w(a_k(x,\xi))]v\|_{L_x^2}\lesssim \|\mathcal{L} v\|_{L_x^2}+\|v\|_{L_x^2},\ k=1,2.\label{2171}
	\end{equation}
\end{lem}
\begin{proof}
	We give only the proof for  $k=1$; the case for  $k=2$ can be treated similarly. 	
	An application of Proposition \ref{premainder} yields
	\begin{eqnarray}
		a_{1}(x,\xi)\sharp (x_1+F_1'(\xi_1))=a_{1}(x,\xi)(x_1+F_1'(\xi_1))+\frac{ih}{2} H_1(x,\xi)+r^{a_{1}(x,\xi)\sharp (x_1+F_1'(\xi_1))},\notag
	\end{eqnarray}
	where $H_1(x,\xi)=\partial_{x_1} a_{1}(x,\xi)F_1''(\xi_1) -\partial_{\xi_1} a_{1}(x,\xi)\in S_0(1)$ by (\ref{F}). For (\ref{l1.1}), it remains to show that   $r^{a_{1}(x,\xi)\sharp (x_1+F_1'(\xi_1))}\in h^2S_0(1)$. By $	\int _{\mathbb{R}^2}e^{-\frac{2i}{h}y\cdot \zeta}dy =(\pi h)^2\delta (\zeta)$, we get
	\begin{eqnarray}
		&&r^{a_{1}(x,\xi)\sharp (x_1+F_1'(\xi_1))}\notag\\
		&=& \frac{h^2}{4}\frac{1}{(\pi h)^4} \int _{\mathbb{R}^2} \int_{\mathbb{R}^2}\int _{\mathbb{R}^2} \int _{\mathbb{R}^2}  e^{\frac{2i}{h}(\eta\cdot z-y\cdot \zeta)}\int _0^1\partial_{x_1}^2 a_{1} (x+tz,\xi+t\zeta) (1-t)dt F_1'''(\xi_1+\eta_1) dyd\eta dzd\zeta \notag\\
		&=& \frac{h^2}{4\pi^2}  \int _{\mathbb{R}^2} \int _{\mathbb{R}^2} e^{2i\eta\cdot z} \int _0^1 \partial_{x_1}^2a_{1} (x+t\sqrt {h}z,\xi)  (1-t)dt  F_1'''(\xi_1+\sqrt {h}\eta_1) d\eta dz
	\end{eqnarray}
Using the identity
\begin{equation}
	\left( \frac{1-2iz\cdot \partial_\eta}{1+4|z|^2}\right)^3 \left( \frac{1-2i\eta \cdot \partial_z}{1+4|\eta|^2}\right)^3 e^{2i\eta\cdot z}=e^{2i\eta\cdot z}\notag
\end{equation}
to integrate by parts, we obtain
\begin{equation}
	|r^{a_{1}(x,\xi)\sharp (x_1+F_1'(\xi_1))}| \lesssim h^2\int _{\mathbb{R}^2}\int _{\mathbb{R}^2} <z>^{-3} <\eta>^{-3} d\eta dz \lesssim h^2.\label{ls1}
\end{equation}
	Similarly, we have
	\begin{eqnarray}
		|\partial_{x} ^\alpha \partial_{\xi} ^\beta r^{a_{1}(x,\xi)\sharp (x_1+F_1'(\xi_1))}|\le C_{\alpha ,\beta} h^2\qquad \forall \alpha  ,\ \beta \in \mathbb{N}^2,\label{ls2}
	\end{eqnarray}
	i.e. $r^{a_{1}(x,\xi)\sharp (x_1+F_1'(\xi_1))}\in h^2 S_0(1)$. This completes the proof of (\ref{l1.1}).
	
	Now we  proceed to  prove (\ref{l1.3}).  By  Proposition \ref{premainder}
	\begin{eqnarray}
		a_{1}(x,\xi)\sharp \left(x_1+F_1'(\xi_1)\right)^2 &=& a_{1}(x,\xi) \left(x_1+F_1'(\xi_1)\right)^2 +ihH_2(x,\xi) \left(x_1+F_1'(\xi_1)\right) \notag\\
		&& +r^{a_{1}(x,\xi)\sharp (x_1+F_1'(\xi_1))^2},\notag
	\end{eqnarray}
	where $H_2(x,\xi)=\partial_{x_1} a_{1}(x,\xi)F_1''(\xi_1)-\partial_{\xi_1} a_{1}(x,\xi)\in S_0(1)$ by (\ref{F}). Using (\ref{l1.1}), we can write, for some  symbol $a(x,\xi)\in S_0(1)$
	\begin{equation}
		ih H_2(x,\xi)( x_1+F_1'(\xi_1))=ih H_2(x,\xi)\sharp (x_1+F_1'(\xi_1)) +h^2a(x,\xi).\label{899}
	\end{equation}
	Therefore, for (\ref{l1.3}) it suffices  to show that  $r_1^{a_{1}(x,\xi)\sharp (x_1+F_1'(\xi_1))^2}$  can be written as the form of (\ref{l1.3}). By (\ref{1}), we have
	\begin{eqnarray}
		&&	r_1^{a_{1}(x,\xi)\sharp (x_1+F_1'(\xi_1))^2} \notag\\
		&=&  -\frac{h^2}{2} \frac{1}{(\pi h)^4} \int _{\mathbb{R}^2} \int_{\mathbb{R}^2}\int _{\mathbb{R}^2} \int _{\mathbb{R}^2}  e^{\frac{2i}{h}(\eta\cdot z-y\cdot \zeta)}\int _0^1 \partial_{x_1}^2 a_{1} (x+tz,\xi+t\zeta)  (1-t)dt    \notag\\
		&&  \times \left[ \left(F_1''(\xi_1+\eta_1)\right)^2+\left(x_1+y_1+F_1'(\xi_1+\eta_1)\right)F_1'''(\xi_1+\eta_1)\right]  dyd\eta dzd\zeta \notag\\
		&&+  \frac{h^2}{(\pi h)^4} \int _{\mathbb{R}^2} \int_{\mathbb{R}^2}\int _{\mathbb{R}^2} \int _{\mathbb{R}^2}  e^{\frac{2i}{h}(\eta\cdot z-y\cdot \zeta)}\int _0^1 \partial_{x_1}\partial_{\xi_1} a_{1} (x+tz,\xi+t\zeta)   (1-t)dt F_1''(\xi_1+\eta_1) dyd\eta dzd\zeta\notag\\
		&&-\frac{h^2}{2} \frac{1}{(\pi h)^4} \int _{\mathbb{R}^2} \int_{\mathbb{R}^2}\int _{\mathbb{R}^2} \int _{\mathbb{R}^2}  e^{\frac{2i}{h}(\eta\cdot z-y\cdot \zeta)}\int _0^1 \partial_{\xi_1}^2 a_{1} (x+tz,\xi+t\zeta)  (1-t)dt dyd\eta dzd\zeta.\label{1132}
	\end{eqnarray}
It is easy to see that the  right-hand side of (\ref{1132})  belongs to the symbol class  $h^2S_0(1)$ by the similar argument as in the proof of (\ref{ls1})--(\ref{ls2}), except the term
\begin{eqnarray}
	 R(x,\xi)&=&-\frac{h^2}{2} \frac{1}{(\pi h)^4}  \int _{\mathbb{R}^8}  e^{\frac{2i}{h}(\eta\cdot z-y\cdot \zeta)}\int _0^1 \partial_{x_1}^2 a_{1} (x+tz,\xi+t\zeta)  (1-t)dt  \notag\\ && \times \left(x_1+y_1+F_1'(\xi_1+\eta_1)\right)F_1'''(\xi_1+\eta_1)  dyd\eta dzd\zeta.\notag
\end{eqnarray}
By Mean Value Theorem
\begin{equation}
	x_1+y_1+F_1'(\xi_1+\eta_1)=x_1+F_1'(\xi_1)+\eta_1 \int_0^1 F_1''(\xi_1+s\eta_1)ds +y_1, \notag
\end{equation}
we decompose  $R(x,\xi)$ accordingly into three parts
\begin{equation}
	R(x,\xi)=A(x,\xi)(x_1+F_1'(\xi_1))+R_2(x,\xi)+R_3(x,\xi).\notag
\end{equation}
Integrate  $R_2$ via  the identity    $\eta_1 e^{\frac{2i}{h}(\eta\cdot z-y\cdot\zeta)}=\frac{h}{2i} \partial_{z_1}e^{\frac{2i}{h}(\eta\cdot z-y\cdot\zeta)}$ to  get
\begin{eqnarray}
		R_{2}(x,\xi)
	&=& \frac{h^3}{4i} \frac{1}{(\pi h)^4} \int _{\mathbb{R}^2} \int_{\mathbb{R}^2}\int _{\mathbb{R}^2} \int _{\mathbb{R}^2}  e^{\frac{2i}{h}(\eta\cdot z-y\cdot \zeta)}\int _0^1 \partial_{x_1}^3 a_{1} (x+tz,\xi+t\zeta)  t(1-t)dt   \notag\\
	&&\qquad \times  F_1'''(\xi_1+\eta_1) \int _0^1 F_1''(\xi_1+s\eta_1) dsdyd\eta dzd\zeta ,\notag
\end{eqnarray}
	which belongs to the symbol class  $h^3S_0(1)$ by the similar argument as  in the proof of (\ref{ls1})--(\ref{ls2}). Moreover, a similar argument shows   that $A(x,\xi),\ R_3(x,\xi)\in h^2S_0(1)$; so that by (\ref{l1.1})
	\begin{equation}
		A(x,\xi)(x_1+F_1'(\xi_1))+R_2(x,\xi)+R_3(x,\xi)=h\widetilde{a}(x,\xi)\sharp (x_1+F_1'(\xi_1))+h^2 \widetilde{b}(x,\xi),\notag
	\end{equation}
for some symbols  $\widetilde{a}(x,\xi),\ \widetilde{b}(x,\xi)\in S_0(1)$.  	This finishes the proof of (\ref{l1.3}).

Similarly, by calculating  $(x_k+F'_k(x_k))\sharp a_k(x,\xi)$ and  $(x_k+F'_k(x_k))^2\sharp a_k(x,\xi)$, we get
\begin{equation}
	a_k(x,\xi)(x_k+F'_k(x_k))= (x_k+F'_k(x_k))\sharp a_k(x,\xi)+h\widetilde{b}_k(x,\xi),\label{2172}
\end{equation}
\begin{equation}
	a_k(x,\xi)(x_k+F'_k(x_k))^2= (x_k+F'_k(x_k))^2\sharp a_k(x,\xi)+h\widetilde{c}_k(x,\xi)\sharp (x_k+F'_k(x_k))+h^2\widetilde{d}_k(x,\xi),\label{2173}
\end{equation}
for  some symbols  $\widetilde{b}_k(x,\xi),\ \widetilde{c}_k(x,\xi),\ \widetilde{d}_k(x,\xi)\in S_0(1)$. Combining (\ref{l1.1})--(\ref{l1.3}) and (\ref{2172})--(\ref{2173}), we obtain
\begin{equation}
	(x_k+F_k'(\xi_k))\sharp a_k-a_k\sharp (x_k+F_k'(\xi_k))=h(b_k-\widetilde{b}_k),\notag
\end{equation}
\begin{equation}
	(x_k+F_k'(\xi_k))^2\sharp a_k-a_k\sharp (x_k+F_k'(\xi_k))^2=h\left[c_k\sharp (x_k+F_k'(\xi_k))-\widetilde{c}_k\sharp (x_k+F_k'(\xi_k))\right]+h^2(d_k-\widetilde{d}_k).\notag
\end{equation}
Taking the Weyl quantization, recalling the definition (\ref{L}), it follows that
$$
	[\mathcal{L} _k,G_h^w(a_k)]v=G_h^w(b_k-\widetilde{b}_k)v,$$
   $$[\mathcal{L} _k^2,G_h^w(a_k)]v=G_h^w(c_k)\circ \mathcal{L} _kv-G_h^w(\widetilde{c}_k)\circ \mathcal{L} _kv+G_h^w(d_k-\widetilde{d}_k)v,\notag
$$
for  $k=1,2.$ An application of Proposition \ref{pjl2} then yields the desired estimates (\ref{2171}).
\end{proof}

Using the same argument as in the proof of Lemma \ref{l1}, we get the following  lemma easily and omit the details.
\begin{lem}\label{l11}
	Suppose that  $a_{k}(x,\xi)\in S_0(1),\ 1\le k\le2$. There exist symbols $b_{k}(x,\xi)$, $c_{k}(x,\xi)$, $d_{k}(x,\xi)\in S_0(1), \ 1\le k\le2$ such that
	\begin{eqnarray}
		a_{k}(x,\xi)\left(\xi_k- d\phi_k(x_k) \right)=a_{k}(x,\xi)\sharp (\xi_k-d\phi_k(x_k) )+hb_{k}, \label{l1.2}
	\end{eqnarray}
	and
	\begin{eqnarray}
		a_{k}(x,\xi) \left(\xi_k-d\phi_k(x_k) \right)^2 =a_{k}(x,\xi)\sharp \left(\xi_k-d\phi_k(x_k) \right)^2+h c_{k}\sharp (\xi_k-d\phi_k(x_k))
		+h^2d_{k}, \label{l1.4}
	\end{eqnarray}
	for $k=1,2$.
\end{lem}
\begin{lem}\label{l3}
	Suppose that $\Gamma(\xi)$ is a smooth function that satisfies  $\Gamma\equiv0$ in a neighborhood of  zero and $|\partial_\xi ^\alpha \Gamma(\xi)|\le C_\alpha <\xi>^{-2-|\alpha| }$ for any $\alpha \in \mathbb{N}^2$. There exist symbols  $a_{kj}(x,\xi),b_{kj}(x,\xi)\in S_{\frac{1}{2}}(<\frac{x+F'(\xi)}{\sqrt{h}}>^{-2}),1\le k,j\le2$,  such that
	\begin{eqnarray}
		\Gamma (\frac{x+F'(\xi)}{\sqrt{h}})(x_k+F_k'(\xi_k))^2&=&\left(\Gamma (\frac{x+F'(\xi)}{\sqrt{h}}) \sharp (x_k+F_k'(\xi_k))\right) \sharp (x_k+F_k'(\xi_k)) \notag\\
		&&+\sum_{j=1}^{2}\left(ha_{kj}(x,\xi)\sharp (x_j+F_j'(\xi_j))+h^2b_{kj}(x,\xi)\right),\notag
	\end{eqnarray}
	for $k=1,2$.
\end{lem}
\begin{proof}
	We consider only the case $k=1$; the case $k=2$ follows from a similar argument.  An application of  Proposition \ref{premainder} yields
	\begin{eqnarray}
		&& \left(\Gamma (\frac{x+F'(\xi)}{\sqrt{h}})\sharp (x_1+F_1'(\xi_1))\right)\sharp (x_1+F_1'(\xi_1)) \notag\\
		&=& \left(\Gamma(\frac{x+F'(\xi)}{\sqrt{h}})(x_1+F_1'(\xi_1))\right)\sharp (x_1+F_1'(\xi_1))-\frac{h}{4}a_0(x,\xi) \sharp (x_1+F_1'(\xi_1)),  \label{874}
	\end{eqnarray}
	where
	\begin{eqnarray}
		&&a_0(x,\xi)\notag\\
		&=&\frac{1}{(\pi h)^4} \int _{\mathbb{R}^2} \int_{\mathbb{R}^2}\int _{\mathbb{R}^2} \int _{\mathbb{R}^2}  e^{\frac{2i}{h}(\eta\cdot z-y\cdot \zeta)}  \int _0^1 (\partial_{\xi_1}^2\Gamma)(\frac{x+tz+F'(\xi+t\zeta)}{\sqrt {h}}) (1-t)dt \notag\\
&&\times F_1'''(\xi_1+\eta_1) dyd\eta dzd\zeta \notag \\
		&=& \frac{1}{(\pi h)^2}\int _{\mathbb{R}^2}\int _{\mathbb{R}^2}e^{\frac{2i}{h}\eta\cdot z}\int_0^1 (\partial_{\xi_1}^2\Gamma)(\frac{x+tz+F'(\xi)}{\sqrt {h}})(1-t)dtF_1'''(\xi_1+\eta_1)d\eta dz .\label{2122}
	\end{eqnarray}
We claim that
\begin{eqnarray}
		a_0(x,\xi)&=&\frac{1}{(\pi h)^2}\int _{\mathbb{R}^2}\int _{\mathbb{R}^2}e^{\frac{2i}{h}\eta\cdot z}\int_0^1 (\partial_{\xi_1}^2\Gamma)(\frac{x+tz+F'(\xi)}{\sqrt {h}})(1-t)dtF_1'''(\xi_1+\eta_1)d\eta dz \notag \\
		&\in& S_{\frac{1}{2}}(<\frac{x+F'(\xi)}{\sqrt{h}}>^{-2}).\label{2123}
\end{eqnarray}
	 By making a change of variables and then integrating by parts, we get
	\begin{eqnarray}
		|a_0 (x,\xi)|
		&=&|\frac{1}{\pi^2}\int _{\mathbb{R}^2} \int _{\mathbb{R}^2} \left( \frac{1-2iz\cdot \partial_\eta}{1+4|z|^2}\right)^5 \left( \frac{1-2i\eta \cdot \partial_z}{1+4|\eta|^2}\right)^5 e^{2i\eta\cdot z} \notag\\
		&&\qquad \times \int_0^1(\partial_{\xi_1}^2\Gamma)(\frac{x+F'(\xi)}{\sqrt {h}}+tz)(1-t)dt F_1'''(\xi_1+\sqrt {h}\eta_1)d\eta dz| \notag\\
		&\lesssim &   \int _{\mathbb{R}^2}\int _{\mathbb{R}^2} <z>^{-5}<\eta>^{-5} \int_0^1<\frac{x+F'(\xi)}{\sqrt{h}}+tz>^{-2}(1-t)dtd\eta dz \notag\\
		&\lesssim & \int _{\mathbb{R}^2}\int _{\mathbb{R}^2}<z>^{-3}<\eta>^{-5} <\frac{x+F'(\xi)}{\sqrt{h}}>^{-2}  d\eta dz \lesssim  <\frac{x+F'(\xi)}{\sqrt{h}}>^{-2},\label{ls3}
	\end{eqnarray}
	where in the second inequality we used the elementary inequality
	\begin{equation}
		<x+y>^{-2}\lesssim <x>^{2}<y>^{-2},\qquad \forall x,\ y\in \mathbb{R}^2.\notag
	\end{equation}
	Similarly, we have
	\begin{equation}
		|\partial_x^\alpha \partial_\xi ^\beta a_0(x,\xi) |\le C_{\alpha ,\beta} h^{-\frac{1}{2}(|\alpha |+|\beta|)} <\frac{x+F'(\xi)}{\sqrt{h}}>^{-2} \qquad \forall \alpha ,\ \beta \in \mathbb{N}^2,\label{ls4}
	\end{equation}
	i.e. $a_{0}(x,\xi)\in S_{\frac{1}{2}}(<\frac{x+F'(\xi)}{\sqrt{h}}>^{-2})$. This finishes the proof of (\ref{2123}).

	Next, we deal with the first term on the right-hand side of (\ref{874}).  Setting $\Gamma^{0}(\xi)=\Gamma (\xi)\xi_1$, we deduce from   Proposition \ref{premainder} that
	\begin{eqnarray}
		&& \left(\Gamma(\frac{x+F'(\xi)}{\sqrt{h}})(x_1+F_1'(\xi_1))\right)\sharp (x_1+F_1'(\xi_1))\notag\\
		&=& \Gamma(\frac{x+F'(\xi)}{\sqrt{h}})(x_1+F_1'(\xi_1))^2-\frac{h^{\frac{3}{2}}}{4} \frac{1}{(\pi h)^4} \int _{\mathbb{R}^2} \int_{\mathbb{R}^2}\int _{\mathbb{R}^2} \int _{\mathbb{R}^2}  e^{\frac{2i}{h}(\eta\cdot z-y\cdot \zeta)} \notag\\
		&&\qquad \times \int _0^1 (\partial_{\xi_1}^2 \Gamma^0)(\frac{x+tz+F'(\xi+t\zeta)}{\sqrt {h}})  (1-t)dt  F_1''' (\xi_1+\eta_1)dyd\eta dzd\zeta\notag\\
		&=&\Gamma(\frac{x+F'(\xi)}{\sqrt{h}})(x_1+F_1'(\xi_1))^2 -\frac{h^{\frac{3}{2}}}{4(\pi h)^2}\int _{\mathbb{R}^2}\int _{\mathbb{R}^2}e^{\frac{2i}{h}\eta\cdot z}\int_{0}^{1}(\partial_{\xi_1}^2 \Gamma^0)(\frac{x+tz+F'(\xi)}{\sqrt {h}})(1-t) \mathrm{d}t \notag\\
		&&\qquad \times F_1'''(\xi_1+\eta_1)d\eta dz.  \label{71310}
	\end{eqnarray}
Assume for a while that we have proved
\begin{eqnarray}
	&&\frac{h^{\frac{3}{2}}}{(\pi h)^2}\int _{\mathbb{R}^2}\int _{\mathbb{R}^2}e^{\frac{2i}{h}\eta\cdot z}\int_{0}^{1}(\partial_{\xi_1} ^2\Gamma^0)(\frac{x+tz+F'(\xi)}{\sqrt {h}})(1-t) \mathrm{d}tF_1'''(\xi_1+\eta_1)d\eta dz\notag\\
	&=& \sum_{j=1}^{2}\left(h\widetilde{a}_j(x,\xi)\sharp (x_j+F_j'(\xi_j))+h^2\widetilde{b}_j(x,\xi)\right),\label{281}
\end{eqnarray}
	for some symbols  $\widetilde{a}_j(x,\xi),\ \widetilde{b}_j(x,\xi)\in S_{\frac{1}{2}}(<\frac{x+F'(\xi)}{\sqrt{h}}>^{-2}),\ j=1,2$. Then Lemma \ref{l3} follows by substituting  (\ref{2123}), (\ref{71310}) and (\ref{281}) into (\ref{874}).
	
	In what follows, we prove (\ref{281}). Let  $\Gamma^j(\xi)=\frac{\partial_{\xi_1}^2\Gamma^0(\xi)}{|\xi|^2}\xi_j, j=1,2$, satisfying  $|\partial_\xi ^\alpha \Gamma^j(\xi)|\le C_\alpha <\xi>^{-2-|\alpha |}$ for any $\alpha \in \mathbb{N}^2$.
	By direct computation,  we have
\begin{eqnarray}
	&&\frac{h^{\frac{3}{2}}}{(\pi h)^2}\int _{\mathbb{R}^2}\int _{\mathbb{R}^2}e^{\frac{2i}{h}\eta\cdot z}\int_{0}^{1}(\partial_{\xi_1}^2 \Gamma^0)(\frac{x+tz+F'(\xi)}{\sqrt {h}})(1-t) \mathrm{d}tF_1'''(\xi_1+\eta_1)d\eta dz\notag\\
	&=& \frac{h}{(\pi h)^2}\sum_{j=1}^{2} \left[\int _{\mathbb{R}^2} \int _{\mathbb{R}^2}  e^{\frac{2i}{h}\eta\cdot z} \int_{0}^{1} \Gamma^j(\frac{x+tz+F'(\xi)}{\sqrt {h}})(1-t) \mathrm{d}tF_1'''(\xi_1+\eta_1)d\eta dz (x_j+F_j'(\xi_j)) \right. \notag\\
	&&  \left. +  \int _{\mathbb{R}^2} \int _{\mathbb{R}^2}  z_je^{\frac{2i}{h}\eta\cdot z} \int_{0}^{1} \Gamma^j(\frac{x+tz+F'(\xi)}{\sqrt {h}})(1-t) t\mathrm{d}tF_1'''(\xi_1+\eta_1)d\eta dz \right]\notag\\
	&=&h\sum_{j=1}^{2} \left[ \frac{1}{(\pi h)^2}\int _{\mathbb{R}^2} \int _{\mathbb{R}^2}  e^{\frac{2i}{h}\eta\cdot z} \int_{0}^{1} \Gamma^j(\frac{x+tz+F'(\xi)}{\sqrt {h}})(1-t) \mathrm{d}tF_1'''(\xi_1+\eta_1)d\eta dz (x_j+F_j'(\xi_j)) \right. \notag\\
	&&  \left. +   \frac{ih}{2}\frac{1}{(\pi h)^2}\int _{\mathbb{R}^2} \int _{\mathbb{R}^2}  e^{\frac{2i}{h}\eta\cdot z} \int_{0}^{1} \Gamma^j(\frac{x+tz+F'(\xi)}{\sqrt {h}})(1-t) t\mathrm{d}tF_1^{(4)}(\xi_1+\eta_1)d\eta dz \right]\notag\\
	&=:&h\sum_{j=1}^{2}\left[A_j(x,\xi)(x_j+F_j'(\xi_j))+hB_j(x,\xi)\right].\label{283}
\end{eqnarray}
 By a similar argument as that used to derive (\ref{2123}), we get that
	\begin{equation}
	A_j(x,\xi)\in S_{\frac{1}{2}}(<\frac{x+F'(\xi)}{\sqrt{h}}>^{-2}),\ B_j(x,\xi) \in  S_{\frac{1}{2}}(<\frac{x+F'(\xi)}{\sqrt{h}}>^{-2}),\ k=1,2. \notag
\end{equation}
This proves (\ref{281}),  from which Lemma \ref{l3} follows.
\end{proof}
\begin{lem}\label{l4}
	Suppose that  $\Gamma(\xi)\in C_0^\infty (\mathbb{R}^2)$ is a smooth function that satisfies  $\Gamma\equiv0$ in a neighborhood of  zero and  $a(x,\xi)\in S_0(1)$.
	Then we have the following estimates
	\begin{equation}
		\|G_h^w\left(\Gamma(\frac{x+F'(\xi)}{\sqrt{h}})a(x,\xi)\right)v\|_{L_x^\infty }\lesssim h^{1/2}(\|\mathcal{L} ^2v\|_{L_x^2}+\|\mathcal{L} v\|_{L_x^2}+\|v\|_{L_x^2}),\notag
	\end{equation}
	\begin{equation}
	\|G_h^w\left(\Gamma(\frac{x+F'(\xi)}{\sqrt{h}})a(x,\xi)\right)v\|_{L_x^2 }\lesssim h(\|\mathcal{L} ^2v\|_{L_x^2}+\|\mathcal{L} v\|_{L_x^2}+\|v\|_{L_x^2}),\notag
\end{equation}
where  $\mathcal{L} ^2=(\mathcal{L} ^2_1,\mathcal{L} ^2_2)$ with  $\mathcal{L} _k,\ k=1,2 $ defined by (\ref{L}).
\end{lem}
\begin{proof}
Assume for a while that we have proved: there exist smooth functions  $\widetilde{\Gamma}_j(\xi)\in C_0^\infty (\mathbb{R}^2)$  that satisfies  $\widetilde{\Gamma}_j\equiv0$ in a neighborhood of  zero, and symbols  $a_j(x,\xi)\in S_0(1)$,    $b_{j k}(x,\xi)$, $  c_j(x,\xi)$, $ c(x,\xi)\in S_{\frac{1}{2}}(<\frac{x+F'(\xi)}{\sqrt{h}}>^{-2}), \ 1\le j\le 3,\ 1\le k\le2$  such that
\begin{eqnarray}
	\Gamma(\frac{x+F'(\xi)}{\sqrt{h}})a(x,\xi)
	&=&\frac{1}{h}\sum_{j=1}^{3}\sum_{k=1}^{2}\widetilde{\Gamma}_j(\frac{x+F'(\xi)}{\sqrt{h}})\sharp (x_k+F_k'(\xi_k))\sharp (x_k+F'_k(\xi_k))\sharp a_j\notag\\
	&&+\sum_{j=1}^{3}\sum_{k=1}^{2}b_{j k}\sharp (x_k+F'_k(\xi_k))\sharp a_j+h\sum_{j=1}^{3}c_j\sharp a_j+hc.\label{2141}
\end{eqnarray}
Taking the Weyl quantization in (\ref{2141}), we get
\begin{eqnarray}
	&&G_h^w(	\Gamma(\frac{x+F'(\xi)}{\sqrt{h}})a(x,\xi))v\notag\\
	&=&h\sum_{j=1}^{3}\sum_{k=1}^{2}G_h^w(\widetilde{\Gamma}_j(\frac{x+F'(\xi)}{\sqrt{h}}))\circ \mathcal{L} ^2_k \circ G_h^w(a_j)v+h\sum_{j=1}^{3}\sum_{k=1}^{2}G_h^w(b_{j k})\circ \mathcal{L} _k\circ  G_h^w(a_j)v\notag\\
		&&+h\sum_{j=1}^{3}G_h^w(c_j)\circ  G_h^w(a_j)v+hG_h^w(c)v.\notag
\end{eqnarray}
Therefore,  using Proposition \ref{pjl2} and Proposition \ref{pj1}, we estimate
	\begin{eqnarray*}
	&&\|G_h^w\left(\Gamma(\frac{x+F'(\xi)}{\sqrt{h}})a(x,\xi)\right)v\|_{L_x^\infty }\\
&\lesssim& h^{1/2}\sum_{j=1}^{3}\sum_{k=1}^{2}\left(\|\mathcal{L} _k^2\circ G_h^w(a_j)v\|_{L_x^2}+\|\mathcal{L} _k\circ G_h^w(a_j)v\|_{L_x^2}\right)+h^{1/2}\|v\|_{L_x^2},\notag
\end{eqnarray*}
\begin{eqnarray*}
	&&\|G_h^w\left(\Gamma(\frac{x+F'(\xi)}{\sqrt{h}})a(x,\xi)\right)v\|_{L_x^2 }\\
&\lesssim &h\sum_{j=1}^{3}\sum_{k=1}^{2}\left(\|\mathcal{L} _k^2\circ G_h^w(a_j)v\|_{L_x^2}+\|\mathcal{L} _k\circ G_h^w(a_j)v\|_{L_x^2}\right)+h\|v\|_{L_x^2},\notag
\end{eqnarray*}
which together with (\ref{2171}) yields the desired estimates in Lemma \ref{l4}.

We now prove (\ref{2141}).  By Proposition \ref{premainder}, we have
\begin{eqnarray}
		&&{\Gamma}(\frac{x+F'(\xi)}{\sqrt{h}})\sharp a(x,\xi)\notag\\
		&=&{\Gamma}(\frac{x+F'(\xi)}{\sqrt{h}}) a(x,\xi)+\frac{i\sqrt {h}}{2}\sum_{k=1}^{2}(\partial_{\xi_k}\Gamma )(\frac{x+F'(\xi)}{\sqrt{h}})E_k(x,\xi)+r^{\Gamma(\frac{x+F'(\xi)}{\sqrt{h}})\sharp a(x,\xi)},\label{2131}
\end{eqnarray}
where  $E_k(x,\xi)=\partial_{\xi_k}a(x,\xi)-\partial_{x_k}a(x,\xi)F_k''(\xi_k)\in S_0(1)$ by (\ref{F}).
Using Proposition \ref{premainder} again, we have, for some symbols  $H_{kj}(x,\xi)\in S_0(1)$,  $1\le k,j\le 2$,
\begin{eqnarray}
	&&(\partial_{\xi_k}\Gamma )(\frac{x+F'(\xi)}{\sqrt{h}})\sharp E_k(x,\xi)\label{2132}\\
&=&(\partial_{\xi_k}\Gamma)(\frac{x+F'(\xi)}{\sqrt{h}}) E_k(x,\xi)+\frac{i\sqrt {h}}{2}\sum_{j=1}^{2}(\partial_{\xi_k} \partial_{\xi_j}\Gamma )(\frac{x+F'(\xi)}{\sqrt{h}})H_{kj}(x,\xi)
	+r^{(\partial_{\xi_k}\Gamma)(\frac{x+F'(\xi)}{\sqrt{h}})\sharp E_k(x,\xi)}.\notag
\end{eqnarray}
Substitution of (\ref{2132}) into (\ref{2131}) yields 	
\begin{eqnarray}
	\Gamma(\frac{x+F'(\xi)}{\sqrt{h}})a(x,\xi)&=& \Gamma(\frac{x+F'(\xi)}{\sqrt{h}})\sharp a(x,\xi)-\frac{i\sqrt {h}}{2}\sum_{k=1}^{2}(\partial_{\xi_k}\Gamma)(\frac{x+F'(\xi)}{\sqrt{h}})\sharp E_k(x,\xi)\notag\\
	&&-\frac{h}{4}\sum_{j=1}^{2}\sum_{k=1}^{2}(\partial_{\xi_k}\partial_{\xi_j}\Gamma)(\frac{x+F'(\xi)}{\sqrt{h}})H_{kj}(x,\xi)-r^{\Gamma(\frac{x+F'(\xi)}{\sqrt{h}})\sharp a(x,\xi)}\notag\\
	&&+\frac{i\sqrt {h}}{2}\sum_{k=1}^{2}r^{(\partial_{\xi_k}\Gamma)(\frac{x+F'(\xi)}{\sqrt{h}})\sharp E_k(x,\xi)}.\label{2134}
\end{eqnarray}
Since  $\Gamma\equiv0$ in a neighborhood of  zero and  $H_{kj}\in S_0(1)$, we have
\begin{equation}
	(\partial_{\xi_k}\partial_{\xi_j}\Gamma)(\frac{x+F'(\xi)}{\sqrt{h}})H_{kj}(x,\xi)\in S_{\frac{1}{2}}(<\frac{x+F'(\xi)}{\sqrt{h}}>^{-2}),\qquad1\le k,j\le2. \label{2135}
\end{equation}
Moreover, by the definition in (\ref{1}) and the similar argument as in the proof of (\ref{ls3})--(\ref{ls4}), we get
\begin{equation}
	r^{\Gamma(\frac{x+F'(\xi)}{\sqrt{h}})\sharp a(x,\xi)},\ r^{(\partial_{\xi_k}\Gamma)(\frac{x+F'(\xi)}{\sqrt{h}})\sharp E_k(x,\xi)} \in hS_{\frac{1}{2}}(<\frac{x+F'(\xi)}{\sqrt{h}}>^{-2}),\ k=1,2. \label{2136}
\end{equation}
Substitution of (\ref{2135}) and (\ref{2136}) into (\ref{2134}) yields
	\begin{equation}
		{	\Gamma}(\frac{x+F'(\xi)}{\sqrt{h}})a(x,\xi)=\sum_{j=1}^{3}{\Gamma}_j(\frac{x+F'(\xi)}{\sqrt{h}})\sharp a_j(x,\xi)+h{c}(x,\xi),\label{2130}
	\end{equation}
for some smooth functions  $\Gamma_j(\xi)\in C_0^\infty (\mathbb{R}^2)$  that satisfies  $\Gamma_j\equiv0$ in a neighborhood of  zero, and symbols  $a_j(x,\xi)\in S_0(1), \ {c}(x,\xi)\in S_{\frac{1}{2}}(<\frac{x+F'(\xi)}{\sqrt{h}}>^{-2})$.

Let   $\widetilde{\Gamma}_j(\xi)=\frac{\Gamma_j(\xi)}{|\xi|^2}, j=1,2$,  satisfying  $|\partial_\xi ^\alpha \widetilde{\Gamma}_j(\xi)|\le C_\alpha  <\xi>^{-2-|\alpha |}$ for any  $\alpha \in \mathbb{N}^2$.  By Lemma \ref{l3}, there are symbols  $b_{j k}(x,\xi),\ c_j(x,\xi)\in S_{\frac{1}{2}}(<\frac{x+F'(\xi)}{\sqrt{h}}>^{-2})$ such that
\begin{eqnarray}
&&	\Gamma_j(\frac{x+F'(\xi)}{\sqrt{h}})=\frac{1}{h}\sum_{k=1}^{2}\widetilde{\Gamma}_j(\frac{x+F'(\xi)}{\sqrt{h}})(x_k+F_k'(\xi_k))^2\label{2133}\\
	&=& \frac{1}{h}\sum_{k=1}^{2}\left(\widetilde{\Gamma}_j(\frac{x+F'(\xi)}{\sqrt{h}})\sharp (x_k+F_k'(\xi_k))\right)\sharp (x_k+F_k'(\xi_k))+\sum_{k=1}^{2}b_{j k}\sharp (x_k+F_k'(\xi_k))+hc_j.\notag
\end{eqnarray}
Substitution of  (\ref{2133}) into (\ref{2130}) then gives (\ref{2141}), from which Lemma \ref{l4} follows.

\end{proof}
\begin{lem}\label{sharp1}
	Suppose that $\gamma(\xi)\in C_0^\infty (\mathbb{R}^2)$, satisfying $\gamma\equiv1$ in a neighborhood of zero. There exist  symbols $a_k(x,\xi),b_k(x,\xi)\in S_{\frac{1}{2}}(<\frac{x+F'(\xi)}{\sqrt{h}}>^{-2}),k=1,2$ such that
	\begin{eqnarray}
		r^{(x\cdot\xi+F(\xi))\sharp \gamma(\frac{x+F'(\xi)}{\sqrt{h}})}-r^{\gamma(\frac{x+F'(\xi)}{\sqrt{h}})\sharp (x\cdot \xi+F(\xi))}=\sum_{k=1}^{2}\left(ha_k(x,\xi)\sharp (x_k+F_k'(\xi_k))+h^2b_k(x,\xi)\right).\notag
	\end{eqnarray}
\end{lem}
	\begin{proof}
	For simplicity, we use the notation $\gamma_{ij}(\xi)=\partial^2_{\xi_i\xi_j}\gamma(\xi),\ \gamma_{iii}(\xi)=\partial^3_{\xi_i}\gamma(\xi)$, $1\le i,j\le2$. By Proposition \ref{premainder}
	\begin{eqnarray}
		&& 	r^{(x\cdot\xi+F(\xi))\sharp \gamma(\frac{x+F'(\xi)}{\sqrt{h}})}-r^{\gamma(\frac{x+F'(\xi)}{\sqrt{h}})\sharp (x\cdot \xi+F(\xi))} \notag\\
		&=& \frac{h}{4}\frac{1}{(\pi h)^4} \sum_{k=1}^{2}\int _{\mathbb{R}^2} \int_{\mathbb{R}^2}\int _{\mathbb{R}^2} \int _{\mathbb{R}^2}  e^{\frac{2i}{h}(\eta\cdot z-y\cdot \zeta)}\left[-\int _0^1F_k''(\xi_k+t\zeta_k) (1-t)dt \right.\notag\\
		&&\left. \times  \gamma_{kk}(\frac{x+y+F'(\xi+\eta)}{\sqrt {h}}) +2\int _0^1(1-t)dt  \gamma_{12}(\frac{x+y+F'(\xi+\eta)}{\sqrt {h}})F_k''(\xi_k+\eta_k) \right. \notag\\
		&&\left. +\int _0^1 \gamma_{kk}(\frac{x+tz+F'(\xi+t\zeta)}{\sqrt {h}})(1-t)dt F_k''(\xi_k+\eta_k)\right. \notag\\
		&&\left. -2\int _0^1 \gamma_{12}(\frac{x+tz+F'(\xi+t\zeta)}{\sqrt {h}})F_k''(\xi_k+t\zeta_k) (1-t)dt \right]dyd\eta dzd\zeta \notag\\
		&=& \frac{h}{4}\frac{1}{(\pi h)^2}\sum_{k=1}^{2} \left[-\int _{\mathbb{R}^2}\int _{\mathbb{R}^2}e^{-\frac{2i}{h}y\cdot \zeta}\int_0^1 F_k''(\xi_k+t\zeta_k)(1-t)dt\gamma_{kk}(\frac{x+y+F'(\xi)}{\sqrt{h}})dyd\zeta\right. \notag\\
		&&\left.+ \int _{\mathbb{R}^2}\int _{\mathbb{R}^2}e^{\frac{2i}{h}\eta\cdot z}\int_0^1 \gamma_{kk} (\frac{x+tz+F'(\xi)}{\sqrt{h}})(1-t)dtF_k''(\xi_k+\eta_k)d\eta d z\right]\notag\\
		&=& \frac{h}{4}\frac{1}{(\pi h)^2}\sum_{k=1}^{2} \left(A_k(x,\xi)+B_k(x,\xi)\right),\label{1083}
	\end{eqnarray}
where
\begin{equation}
	A_k=-\int _{\mathbb{R}^2}\int _{\mathbb{R}^2}e^{-\frac{2i}{h}y\cdot \zeta}\int_0^1 \left(F_k''(\xi_k+t\zeta_k)-F_k''(\xi_k)\right)(1-t)dt \gamma_{kk}(\frac{x+y+F'(\xi)}{\sqrt{h}})dyd\zeta,\notag
\end{equation}
\begin{equation}
	B_k=\int _{\mathbb{R}^2}\int _{\mathbb{R}^2}e^{\frac{2i}{h}\eta\cdot z}\int_0^1 \gamma_{kk} (\frac{x+tz+F'(\xi)}{\sqrt{h}})(1-t)dt\left(F_k''(\xi_k+\eta_k)-F_k''(\xi_k)\right)d\eta d \notag. \notag
\end{equation}
For  $A_k$, using the Mean Value Theorem and the identity  $\zeta_k e^{-\frac{2i}{h}y\cdot \zeta}=-\frac{h}{2i}\partial_{y_k} e^{-\frac{2i}{h}y\cdot \zeta}$ to integrate by parts,  we get
	\begin{eqnarray}
	\frac{h}{4}\frac{1}{(\pi h)^2}A_k(x,\xi)&=& \frac{ih^{\frac{3}{2}}}{8(\pi h)^2}\int _{\mathbb{R}^2}  \int _{\mathbb{R}^2}  e^{-\frac{2i}{h} y\cdot \zeta}\int _0^1\int_0^1 F_k'''(\xi_k+st\zeta_k) t(1-t)dsdt\notag\\
	&&\times \gamma_{kkk}(\frac{x+y+F'(\xi)}{\sqrt {h}}) dyd\zeta.\notag
\end{eqnarray}
Using the same method as that used to derive (\ref{281}), we can write, for some symbols  $a_{kj}(x,\xi)$, $ b_{kj}(x,\xi)\in S_{\frac{1}{2}}(<\frac{x+F'(\xi)}{\sqrt{h}}>^{-2}) $,  $1\le j\le2$
\begin{eqnarray}
	\frac{h}{4}\frac{1}{(\pi h)^2}A_k(x,\xi)=\sum_{j=1}^{2}\left(ha_{kj}(x,\xi)\sharp (x_j+F_j'(\xi_j))+h^2b_{kj}(x,\xi)\right).\label{1084}
\end{eqnarray}
For  $B_k$, we have a similar representation  like (\ref{1084}).
 Lemma \ref{sharp1} then follows by substituting  the representation of  $A_k,\ B_k$ into (\ref{1083}).
	\end{proof}

\begin{lem}
	\label{sharp2}
	Suppose that $\gamma(\xi)\in C_0^\infty (\mathbb{R}^2)$, satisfying $\gamma\equiv1$ in a neighborhood of zero. There exist symbols  $a_k(x,\xi),b_k(x,\xi)\in S_{\frac{1}{2}}(<\frac{x+F'(\xi)}{\sqrt{h}}>^{-2})$, $k=1,2$ such that
	\begin{eqnarray}
		r^{w(x)\sharp \gamma(\frac{x+F'(\xi)}{\sqrt{h}})}-r^{\gamma(\frac{x+F'(\xi)}{\sqrt{h}})\sharp w(x)} =\sum_{ k=1}^{2} \left(ha_k(x,\xi)\sharp (x_k+F_k'(\xi_k))+h^2b_k(x,\xi)\right),\notag
	\end{eqnarray}
	where $w(x)=x\cdot d\phi(x)+F(d\phi(x))$.
\end{lem}
\begin{proof}
	Set $\Gamma^j_k(x,\xi)=\partial_{\xi_k}^j\left(\gamma(\frac{x+F'(\xi)}{\sqrt{h}})\right),k=1,2,\ j\in \mathbb{N}$. Since
	\begin{equation}
		w(x)=x\cdot d\phi(x)+F(d\phi(x))=\sum_{k=1}^{2}\left(x_kd\phi_k(x_k)+F_k(d\phi_k(x_k))\right),\notag
	\end{equation}
we have by (\ref{221})
\begin{equation}
	\partial_{x_k}w(x)=d\phi_k(x_k)+(x_k+F'_k(d\phi_k(x_k)))d^2\phi_k(x_k)=d\phi_k(x_k);\label{w1}
\end{equation}
so that
\begin{equation}\label{w2}
\partial_{x_1}\partial_{x_2}w(x)=0,\qquad \partial_{x_k}\partial_{x_k}w(x)=d^2\phi_k(x_k),\qquad k=1,2.
\end{equation}
It then follows from  Proposition \ref{premainder} that
	\begin{eqnarray}
	&&r^{w(x)\sharp \gamma(\frac{x+F'(\xi)}{\sqrt{h}})}-r^{\gamma(\frac{x+F'(\xi)}{\sqrt{h}})\sharp w(x)}  \notag\\
	&=& \frac{h^2}{4}\frac{1}{(\pi h)^4} \sum_{k=1}^{2}  \int _{\mathbb{R}^2} \int_{\mathbb{R}^2}\int _{\mathbb{R}^2} \int _{\mathbb{R}^2}  e^{\frac{2i}{h}(\eta\cdot z-y\cdot \zeta)}\left[-\int _0^1 d^2\phi_k(x_k+tz_k)(1-t)dt\right.  \notag\\
	&&\left.\times\Gamma^2_k(x+y,\xi+\eta)\int _0^1\Gamma^2_k(x+tz,\xi+t\zeta) (1-t)dt d^2\phi_k(x_k+y_k) dyd\eta dzd\zeta\right] \notag\\
	&=& \frac{h^2}{4}\frac{1}{(\pi h)^2} \sum_{k=1}^{2} \left[-\int _{\mathbb{R}^2} \int _{\mathbb{R}^2}  e^{\frac{2i}{h}\eta\cdot z}\int _0^1 d^2\phi_k(x_k+tz_k)(1-t)dt \Gamma^2_k(x,\xi+\eta)d\eta dz\right. \notag\\
	&&\left.+ \int _{\mathbb{R}^2} \int _{\mathbb{R}^2}  e^{-\frac{2i}{h}y\cdot \zeta}\int _0^1\Gamma^2_k(x,\xi+t\zeta) (1-t)dt d^2\phi_k(x_k+y_k) dyd\zeta\right]\notag\\
	&=&\frac{h^2}{4}\frac{1}{(\pi h)^2}\sum_{k=1}^{2} \left(A_k(x,\xi)+B_k(x,\xi)\right),\label{1085}
\end{eqnarray}
where
\begin{equation}
	A_k=-\int _{\mathbb{R}^2} \int _{\mathbb{R}^2}  e^{\frac{2i}{h}\eta\cdot z}\int _0^1 (d^2\phi_k(x_k+tz_k)-d^2\phi_k(x_k))(1-t)dt \Gamma^2_k(x,\xi+\eta)d\eta dz,\notag
\end{equation}
\begin{equation}
	B_k= \int _{\mathbb{R}^2} \int _{\mathbb{R}^2}  e^{-\frac{2i}{h}y\cdot \zeta}\int _0^1\Gamma^2_k(x,\xi+t\zeta) (1-t)dt (d^2\phi_k(x_k+y_k)-d^2\phi_k(x_k)) dyd\zeta.\notag
\end{equation}
For  $B_k$, using the  Mean Value Theorem and the identity $y_ke^{-\frac{2i}{h}y\cdot \zeta}=-\frac{h}{2i}\partial_{\zeta_k}e^{-\frac{2i}{h}y\cdot \zeta}$ to integrate by part, we obtain
\begin{eqnarray}
	\frac{h^2}{4}\frac{1}{(\pi h)^2}B_k(x,\xi)&=&\int _{\mathbb{R}^2} \int _{\mathbb{R}^2}  e^{-\frac{2i}{h}y\cdot \zeta}\int _0^1h^3\Gamma^3_k(x,\xi+t\zeta) (1-t)tdt \notag\\
	&&\times \int_0^1 d^3\phi_k(x_k+sy_k)ds dyd\zeta,\notag
\end{eqnarray}
where
\begin{eqnarray}
	h^3\Gamma^3_k(x,\xi)&=&h^{\frac{3}{2}}(\partial_{\xi_k}^3\gamma)(\frac{x+F'(\xi)}{\sqrt{h}})(F_k''(\xi_k))^3+3h^2(\partial_{\xi_k}^2\gamma)(\frac{x+F'(\xi)}{\sqrt{h}})F_k''(\xi_k)F_k'''(\xi_k)\notag\\
	&&+h^{\frac{5}{2}}(\partial_{\xi_k}\gamma)(\frac{x+F'(\xi)}{\sqrt{h}})F_k^{(4)}(\xi_k),\ k=1,2. \notag
\end{eqnarray}
Using the same method as that used to derive (\ref{281}), we can write, for some symbols $a_{kj}(x,\xi)$, $b_{kj}(x,\xi)\in S_{\frac{1}{2}}(<\frac{x+F'(\xi)}{\sqrt{h}}>^{-2}) $,  $1\le j\le 2$
\begin{equation}
	\frac{h^2}{4}\frac{1}{(\pi h)^2}B_k(x,\xi)=\sum_{ j=1}^{2} \left(ha_{kj}(x,\xi)\sharp (x_j+F_j'(\xi_j))+h^2b_{kj}(x,\xi)\right).\label{1086}
\end{equation}
For  $A_k$, we have a similar representation like (\ref{1086}).
 Lemma \ref{sharp2} then follows by substituting the representation of  $A_k,\ B_k$ into (\ref{1085}).
\end{proof}

\section{Proof of Theorem \ref{T1}}\label{s3}
This section is devoted to proving Theorem \ref{T1}. It is organized into two subsections. In the first one, we apply the contraction argument and the Strichartz estimate to  prove the  global existence and uniqueness of the solution to (\ref{NLS}). In the second one, we prove the decay estimate (\ref{decay}), combining  the semiclassical analysis method and the ODE argument.
\subsection{Proof of the global existence and uniqueness}
Using the classical energy estimate method, we can obtain the following lemma easily and omit the details.
\begin{lem}\label{zl}
	Suppose that $\text{Im} \lambda \ge0$ and $u$ is a strong $L^2$ solution of (\ref{NLS}) on the time interval $[1,T]$ with $T>1$, then we have
	\begin{equation}
		\|u(t,\cdot)\|_{L^2}\le  \varepsilon\|u_0\|_{L^2}, \ t\in[1,T].\notag
	\end{equation}
\end{lem}
\begin{thm}\label{ztcz}  For any $\varepsilon \in (0,1]$ and any initial datum $u_0\in H^2,\ (x_k+F_k'(D))^2u_0\in L^2,\  k=1,2$  satisfying
	\begin{eqnarray}
		\|u_0\|_{H^2}+\sum_{k=1}^{2} \|(x_k+F_k'(D))^2u_0\|_{L^2}\le1, \label{812z1}
	\end{eqnarray}
	the Cauchy problem (\ref{NLS}) has a unique global solution $u\in C([1 ,\infty); L^2 )\cap L^4_{loc}([1 ,\infty); L^\infty  ) $. Moreover, there exist absolute    constants $T_0,\ C_1>0$ such that
	\begin{equation}
		\|u(t,\cdot)\|_{L^\infty }\le C_1\varepsilon ,\label{7121}
	\end{equation}
	for all $t\in (1,T_0)$.
\end{thm}
\begin{proof}
	Since the proof is classical, we only give a brief here.  Considering the linear inhomogeneous system,
\begin{equation}
		(D_t-F(D))u=f,\ u(1)=\varepsilon u_0,\label{212w1}
\end{equation}
	we have,
	$$
	u =e^{iF(D)t}\varepsilon u_0+i\int^t_1e^{iF(D)(t-s)}f(s)ds.
	$$
	By direct computation, we have
	\begin{equation}
		\|e^{iF(D)t}f\|_{L_x^2}=\|f\|_{L_x^2},\label{5.1}
	\end{equation}
	and
	\begin{eqnarray}
		&&|\mathcal{F}^{-1}(e^{iF(\xi)t})(x)|\notag\\
		&=&|\lim_{\epsilon\rightarrow0}\frac{1}{(2\pi)^2}\int_{\mathbb{R}^2} e^{iF(\xi)t+ix\cdot\xi}e^{-\epsilon |\xi|^2}d\xi|\notag\\
		&=&|\lim_{\epsilon\rightarrow0}\frac{1}{(2\pi)^2{t}}\int_{\mathbb{R}^2} e^{iF(\frac{\eta}{\sqrt{t}})t+i\frac{x\cdot \eta}{\sqrt{t}}-\frac{\epsilon}{t} |\eta|^2}d\eta|\notag\\
		&=&|\lim_{\epsilon\rightarrow0}\frac{1}{(2\pi)^2{t}}\int_{\mathbb{R}^2} e^{-\frac{\epsilon}{t} |\eta|^2}
		\left(\frac{1-i(\sqrt{t}F'(\frac{\eta}{\sqrt{t}})+\frac{x}{\sqrt{t}})\cdot \partial_\eta}{1+|\sqrt{t}F'(\frac{\eta}{\sqrt{t}})+\frac{x}{\sqrt{t}}|^2}
		\right)^3e^{iF(\frac{\eta}{\sqrt{t}})t+i\frac{x\cdot\eta}{\sqrt{t}}}d\eta|\notag\\
		&\lesssim &   \frac{1}{t}\int_{\mathbb{R}^2}<\sqrt{t}F'(\frac{\eta}{\sqrt{t}})+\frac{x}{\sqrt{t}}>^{-3}d\eta \lesssim \frac{1}{t},\label{810s1}
	\end{eqnarray}
	where the last inequality holds since by  (\ref{F}) and (\ref{221}):
	\begin{eqnarray}
		|\sqrt {t} F'(\frac{\eta}{\sqrt {t}})+\frac{x}{\sqrt {t}}|=\sqrt {t}|F'(\frac{\eta}{\sqrt {t}})-F'(d\phi(\frac{x}{t}))|\gtrsim |\eta-\sqrt {t}d\phi(\frac{x}{t})|.\notag
	\end{eqnarray}
	The energy estimate (\ref{5.1}) and the dispersive estimate (\ref{810s1}) imply the Strichartz estimate for the solution $u$ to (\ref{212w1}) (see the theorem of Keel and Tao \cite{Tao}):
	\begin{equation}
		\|u\|_{L^\infty_tL^2_x\cap L^6_tL^3_x}\leq C\varepsilon\|u_0\|_{L^2}
		+C\|f\|_{L^{\frac{6}{5}}_tL^{\frac{3}{2}}_x}.\label{892}
	\end{equation}
	Choosing that $f=\lambda|u|u$, and applying H\"older's inequality,  we have
	\begin{equation}
		\|u\|_{L^\infty_tL^2_x\cap L^6_tL^3_x}\leq C\varepsilon\|u_0\|_{L^2}
		+C\sqrt{t-1}\|u\|_{L^6_tL^3_x}^2.\label{894}
	\end{equation}
	Using the contraction principle in the space $L^\infty _t ([1,T');L^2_x)\cap  L^6 _t ([1,T');L^3_x)$ provided $1<T'<1+(2C)^{-4} \varepsilon ^{-2}$, we obtain that the Cauchy problem (\ref{NLS}) has a unique local solution.  	
	Since  $\|u(t,\cdot)\|_{L^2}\le \varepsilon \|u_0\|_{L^2}$ by Lemma \ref{zl}, this local solution can be extended to  $[0,\infty  ) $  with  $u\in L^\infty _t ([1,T);L^2_x)\cap  L^6 _t ([1,T);L^3_x)$ for all $T>1$. Moreover,  it is easy to check that $u\in C([0,T];L^2)$ for all $T>0$, and omit the details.
	
	It remains to  prove the inequality (\ref{7121}). Notice that
	\begin{equation}
		\partial_{x_k}^2 u =e^{iF(D)t}\varepsilon \partial_{x_k}^2 u_0+i\lambda \int^t_1e^{iF(D)(t-s)}\partial_{x_k}^2(|u|u)ds,\ k=1,2\notag
	\end{equation}
	and $|\partial_{x_k}^2(|u|u)|\lesssim |u||\Delta u|+|\nabla u|^2$ (see (\ref{891})). Using successively Strichartz estimate (\ref{892}) and  H\"older's inequality, we get
	\begin{eqnarray}
		\|\Delta u\|_{L^\infty_tL^2_x\cap L^6_tL^3_x}&\leq& C\varepsilon\|u_0\|_{H^2}
		+C\|u\Delta u\|_{L^{\frac{6}{5}}_tL^{\frac{3}{2}}_x}+C\||\nabla u|^2\|_{L^{\frac{6}{5}}_tL^{\frac{3}{2}}_x}\notag\\
		&\le &   C\varepsilon  \|u_0\|_{H^2} +C\sqrt {t-1}\|u\|_{L^{6}_tL^{3}_x}\|\Delta u\|_{_{L^{6}_tL^{3}_x}}, \notag
	\end{eqnarray}
	where we also used Gagliardo-Nirenberg's inequality  $\||\nabla u|^2\|_{L^{\frac{3}{2}}_x}=\|\nabla u\|_{L^3_x}^2\lesssim \|u\|_{L^3_x}\|\Delta u\|_{L^3_x}$.
	Choosing $T_0>1$ sufficiently approaches $1$, the above inequality implies that
	\begin{equation}
		\|\Delta u\|_{L^\infty ((1,t);L^2)\cap L^6((1,t);L^3)}\le 2C\varepsilon \|u_0\|_{H^2},\qquad \forall t\in (1,T_0).\notag
	\end{equation}
	This together with Lemma \ref{zl} and   Sobolev's embedding $H^2(\mathbb{R}^2)\hookrightarrow L^\infty (\mathbb{R}^2)$ yields the desired estimate (\ref{7121}).
\end{proof}

\subsection{Proof of the global decay estimate}
The goal of this subsection is to derive the decay estimate (\ref{decay}),  thus completing the proof of Theorem \ref{T1}.

Let  $u$ be the solution of (\ref{NLS}) given by Theorem \ref{T3}. We make first a semiclassical change of variables
\begin{equation}\label{4.2}
	u(t,x)=hv(t,hx),\ h=\frac{1}{t},
\end{equation}
that allows rewriting the equation (\ref{NLS}) as
\begin{equation}
	(D_t-G_h^w(x\cdot\xi+F(\xi)))v=\lambda h|v| v.\label{TNLS}
\end{equation}
By direct calculation, we have
\begin{equation}
	\|u(t,\cdot)\|_{L^2}=\|v(t,\cdot)\|_{L^2},\qquad \|u(t,\cdot)\|_{L^\infty }=\sqrt {h} \|v(t,\cdot)\|_{L^\infty }. \label{uv1}
\end{equation}
Moreover,  we have that (recall that   $\mathcal{L} _k$ are the operators defined in (\ref{L}))
\begin{equation}
	h(\mathcal{L} _k^2v)(t,hx)=(x_k+tF_k'(D))^2u(t,x)\qquad k=1,2;\notag
\end{equation}
so that
\begin{equation}
	\|\mathcal{L} ^2v(t,x)\|_{L_x^2}=\sum_{k=1}^{2} \|(x_k+tF_k'(D))^2u\|_{L_x^2} .\label{812z2}
\end{equation}
By (\ref{uv1}), the decay estimate (\ref{decay}) is equivalent to
\begin{equation}
	\|v(t,x)\|_{L^\infty _x}\lesssim \varepsilon ,\qquad t>1.\label{262}
\end{equation}

To prove (\ref{262}), we   decompose  $v=v_\Lambda+v_{\Lambda^c}$ with
\begin{equation}
	v_\Lambda=G^w_h(\Gamma(x,\xi)) v,\notag
\end{equation}
where $\Gamma(x,\xi)=\gamma(\frac{x+F'(\xi)}{\sqrt{h}})$, $\gamma\in C^\infty_0(\mathbb{R}^2)$ satisfying  $\gamma\equiv1$ in a neighbourhood of zero.

We will use the following Sobolev type inequality. For the convenience of the readers, we give the proof in the Appendix.
\begin{lem}\label{l2v}
	Assume $v:[1 , T]\times \mathbb{R}^2\rightarrow \mathbb{C}, T>1$, there exists a positive constant $C_2$ independent of $T$ and $ v$ such that for all $t\in [1 ,T]$
	\begin{equation}
		\|\mathcal{L} ^2(|v| v)\|_{L_x^2}\le C_2 \|v\|_{L_x^\infty } (\|v\|_{L_x^2}+\|\mathcal{L} ^2v\|_{L_x^2}),\label{lvalpha}
	\end{equation}
	and
	\begin{equation}
		\label{751}
		\|\mathcal{L} v\|_{L_x^2} \le C_2( \|v\|_{L_x^2 }+\|\mathcal{L} ^2v\|_{L_x^2}),
	\end{equation}
	provided the right-hand sides are finite.
\end{lem}

The rest of this subsection is organized as follows. In  Lemma \ref{lvc}, we  show  that  $v_{\Lambda ^c}$  decays faster than  $v$ and   (\ref{262}) is reduced to prove
\begin{equation}
	\|v_{\Lambda }(t,x)\|_{L^\infty _x}\lesssim \varepsilon ,\qquad t>1.\notag
\end{equation}
To derive this,   we  apply  $G_h^w(\Gamma(x,\xi))$ to (\ref{TNLS}) and deduce an ODE for  $v_{\Lambda }$:
\begin{equation}
	D_tv_\Lambda=w(x)v_{\Lambda }+\lambda t^{-1}  |v_\Lambda| v_\Lambda+R(v),\label{4.34}
\end{equation}
where $w(x)=x\cdot d\phi(x)+F(d\phi(x))$ and
\begin{eqnarray}
	R(v)&=&[D_t-G^w_h(x\cdot\xi+F(\xi)),G^w_h(\Gamma)]v+G_h^w(x\cdot\xi+F(\xi)-w(x))v_{\Lambda}\notag \\
	&&- \lambda t^{-1}G^w_h(1-\Gamma)(|v| v)+\lambda t^{-1}\left( |v| v-|v_\Lambda| v_\Lambda\right).\label{1131}
\end{eqnarray}
Then we establish the  decay estimates of  $R(v)$ in Lemmas \ref{l4.5}--\ref{l4.4}.  Finally, at the end of this subsection, we  derive the desired  $L^\infty $ estimate for  $v_{\Lambda }$ and then in the solution  $u$, combining the ODE and  the bootstrap argument.  This together with Theorem \ref{T3} finishes the proof of Theorem \ref{T1}.

\begin{lem}\label{lvc}
	Suppose  $u$ is a solution of (\ref{NLS}) given by Theorem \ref{ztcz}	and $v$ is defined by (\ref{4.2}), there is a constant $C_3>0$ such that   for all $t>1$,
	\begin{equation}
		\|v_{\Lambda^c}(t,x)\|_{L_x^\infty}\le C_3 t^{-\frac{1}{2}}(\|v\|_{L_x^2}+\|\mathcal{L} ^2v\|_{L_x^2}),\label{7126}
	\end{equation}
	\begin{equation}
		\|v_{\Lambda^c}(t,x)\|_{L_x^2}\le C_3 t^{-1}(\|v\|_{L_x^2}+\|\mathcal{L} ^2v\|_{L_x^2})\label{7127}.
	\end{equation}
\end{lem}
\begin{proof}
	Set ${\Gamma}_{-2}(\xi)=\frac{1-\gamma(\xi)}{|\xi|^2}$,  satisfying  $|\partial_\xi^\alpha\Gamma_{-2}(\xi)|\le C_\alpha  <\xi>^{-2-|\alpha|}$ for any $\alpha \in \mathbb{N}^2$.  By Lemma \ref{l3}, there are symbols  $a_{kj}(x,\xi),\ b_{kj}(x,\xi)\in S_{\frac{1}{2}}(<\frac{x+F'(\xi)}{\sqrt{h}}>^{-2})$, $1\le k,j\le2$ such that
\begin{eqnarray*}
	&& 1-\gamma(\frac{x+F'(\xi)}{\sqrt{h}} ) = \frac{1}{h} \sum_{k=1}^{2} \Gamma_{-2} (\frac{x+F'(\xi)}{\sqrt{h}}) (x_k+F_k'(\xi_k))^2 \notag\\
	&=& \frac{1}{h} \sum_{k=1}^{2} \left[\left(\Gamma_{-2}(\frac{x+F'(\xi)}{\sqrt{h}})\sharp (x_k+F_k'(\xi_k))\right)\sharp (x_k+F_k'(\xi_k))\right. \notag\\
	&&\left.+\sum_{j=1}^{2}\left(ha_{kj}\sharp (x_j+F_j'(\xi_j)) +h^2b_{kj}\right)\right].   \notag
\end{eqnarray*}
Taking the Weyl quantization, then  using Proposition  \ref{pcomposition}, one obtains
\begin{equation}
	v_{\Lambda ^c}(t,x)=h\sum_{k=1}^{2}\left[
	G_h^w(\Gamma_{-2}(\frac{x+F'(\xi)}{\sqrt{h}}))\circ \mathcal{L} _k^2v+\sum_{j=1}^{2}\left(G_h^w(a_{kj})\circ \mathcal{L} _jv+G_h^w(b_{kj})v\right)\right].\notag
\end{equation}
By Proposition \ref{pj1} and Proposition \ref{pj2}, one has
	\begin{equation}
	\|v_{\Lambda ^c}(t,x)\|_{L_x^\infty } \lesssim  h^{1/2} (\|\mathcal{L} ^2v\|_{L_x^2} +\|\mathcal{L} v\|_{L_x^2} +\|v\|_{L_x^2}).\label{217x1}
\end{equation}
On the other hand,  since  $S_{\frac{1}{2}}(<\frac{x+F'(\xi)}{\sqrt{h}}>^{-2})\subset S_{\frac{1}{2}}(1)$, it follows from Proposition \ref{pjl2} that
	\begin{equation}
		\|v_{\Lambda ^c}(t,x)\|_{L_x^2 } \lesssim  h(\|\mathcal{L} ^2v\|_{L_x^2} +\|\mathcal{L} v\|_{L_x^2} +\|v\|_{L_x^2}).\label{217x2}
	\end{equation}
Substitution of (\ref{751}) into (\ref{217x1})--(\ref{217x2}) yields the desired estimates in  Lemma \ref{lvc}.
\end{proof}

Lemmas \ref{l4.5}--\ref{l4.4} are devoted to prove the decay estimate of  $R(v)$ in (\ref{1131}).
\begin{lem}\label{l4.5}
	With the same assumptions as in Lemma \ref{lvc}, the following inequalities hold for all $t>1$
	\begin{equation}
		\|G^w_h(1-\Gamma)(|v| v)\|_{L_x^\infty}\lesssim  t^{-\frac{1}{2}}\|v\|_{L_x^\infty } (\|v\|_{L_x^2}+\|{\mathcal{L}}^2v\|_{L_x^2}),\notag
	\end{equation}
	\begin{equation}
		\|G^w_h(1-\Gamma)(|v| v)\|_{L_x^2}\lesssim  t^{-1}\|v\|_{L_x^\infty }(\|v\|_{L_x^2}+\|{\mathcal{L}}^2v\|_{L_x^2}).\notag
	\end{equation}
\end{lem}
\begin{proof}
	Using the same method as that used to derive (\ref{7126})--(\ref{7127}), one gets
	\begin{eqnarray}
		\|G_h^w(1-\Gamma)(|v| v)\|_{L_x^\infty } \lesssim t^{-1/2}(\||v| v\|_{L_x^2}+\|\mathcal{L} ^2(|v| v)\|_{L_x^2}),\notag
	\end{eqnarray}
	\begin{eqnarray}
		\|G_h^w(1-\Gamma)(|v| v)\|_{L_x^2 }\lesssim  t^{-1}(\||v| v\|_{L_x^2}+\|\mathcal{L} ^2(|v| v)\|_{L_x^2}),\notag
	\end{eqnarray}
	which together with the inequality  (\ref{lvalpha}) yields the desired estimates in Lemma \ref{l4.5}.
\end{proof}
Similarly, we can obtain the following lemma easily and omit the details.
\begin{lem}\label{l5.6}
	With the same assumptions as in Lemma \ref{lvc}, the following inequalities hold for all $t>1$
	\begin{equation}
		\||v| v-|v_\Lambda| v_\Lambda\|_{L_x^\infty}\lesssim t^{-1/2}(\|v\|_{L_x^\infty }+\|v_{\Lambda }\|_{L_x^\infty }  )(\|v\|_{L_x^2}+\|{\mathcal{L}}^2v\|_{L_x^2}),\notag
	\end{equation}
	\begin{equation}
		\||v| v-|v_\Lambda| v_\Lambda\|_{L_x^2}\lesssim  t^{-1}(\|v\|_{L_x^\infty }+\|v_{\Lambda }\|_{L_x^\infty }  )(\|v\|_{L_x^2}+\|{\mathcal{L}}^2v\|_{L_x^2}).\notag
	\end{equation}
\end{lem}
\begin{lem}\label{l4.3}
	With the same assumptions as in Lemma \ref{lvc}, the following inequalities hold for all $t>1$
	\begin{eqnarray}
		\|[D_t-G^w_h(x\cdot\xi+F(\xi)),G^w_h(\Gamma)]v\|_{L_x^\infty } \lesssim t^{-3/2} (\|v\|_{L_x^2} +\|\mathcal{L} ^2v\|_{L_x^2}),\notag
	\end{eqnarray}
	and
	\begin{eqnarray}
		\|[D_t-G^w_h(x\cdot \xi+F(\xi)),G^w_h(\Gamma)]v\|_{L_x^2 } \lesssim t^{-2} (\|v\|_{L_x^2} +\|\mathcal{L} ^2v\|_{L_x^2}).\notag
	\end{eqnarray}
\end{lem}
\begin{proof}
	First we start by calculating  $[D_t,G_h^w(\Gamma)]=D_tG_h^w(\Gamma)-G_h^w(\Gamma)D_t$.
	Since $h=t^{-1}$, by direct computation, we have, under the notation $\gamma_k(\xi)=\partial_{\xi_k}\gamma(\xi)$,
	\begin{eqnarray}
		&&D_tG^w_h(\Gamma)v\notag\\
		&=& \frac{1}{i}\partial_t  \left[\frac{1}{(2\pi h)^2}\int_{\mathbb{R}^2}\int_{\mathbb{R}^2} e^{\frac{i}{h}(x-y)\cdot \xi} \gamma(\frac{\frac{x+y}{2}+F'(\xi)}{\sqrt {h}})v(t,y)dyd\xi\right]\notag\\
		&=&G_h^w(\Gamma)D_tv-\frac{it}{2\pi^2}\int_{\mathbb{R}^2}\int_{\mathbb{R}^2} e^{\frac{i}{h}(x-y)\cdot \xi} \gamma(\frac{\frac{x+y}{2}+F'(\xi)}{\sqrt {h}})v(t,y)dyd\xi\notag\\
		&&+\frac{1}{(2\pi h)^2}\int_{\mathbb{R}^2}\int_{\mathbb{R}^2} e^{\frac{i}{h}(x-y)\cdot \xi} (x-y)\cdot\xi\gamma(\frac{\frac{x+y}{2}+F'(\xi)}{\sqrt {h}})v(t,y)dyd\xi\notag\\
		&&+\frac{1}{(2\pi h)^2} \sum_{k=1}^{2}\int_{\mathbb{R}^2}\int_{\mathbb{R}^2} e^{\frac{i}{h}(x-y)\cdot \xi}\gamma_k(\frac{\frac{x+y}{2}+F'(\xi)}{\sqrt{h}})
		(\frac{x_k+y_k}{2}+F_k'(\xi_k))\frac{\sqrt{h}}{2i}v(t,y)dyd\xi\notag\\
		&=&-2hiG_h^w(\Gamma)v+\frac{1}{(2\pi h)^2}\int_{\mathbb{R}^2}\int_{\mathbb{R}^2} e^{\frac{i}{h}(x-y)\cdot \xi} (x-y)\cdot\xi\gamma(\frac{\frac{x+y}{2}+F'(\xi)}{\sqrt {h}})v(t,y)dyd\xi\notag\\
		&& -\frac{i\sqrt {h}}{2}\sum_{k=1}^{2} G_h^w\left(\gamma_k(\frac{x+F'(\xi)}{\sqrt{h}})(x_k+F_k'(\xi_k))\right)v+G_h^w(\Gamma)D_tv.\label{812s1}
	\end{eqnarray}
Moreover, using the identity  $(x-y)\cdot \xi e^{\frac{i}{h}(x-y)\cdot\xi}=\frac{h}{i}\sum_{k=1}^{2}\partial_{\xi_k}e^{\frac{i}{h}(x-y)\cdot\xi}\xi_k$ to integrate by parts, we get
	\begin{eqnarray}
		&&\frac{1}{(2\pi h)^2}\int_{\mathbb{R}^2}\int_{\mathbb{R}^2} e^{\frac{i}{h}(x-y)\cdot \xi} (x-y)\cdot\xi\gamma(\frac{\frac{x+y}{2}+F'(\xi)}{\sqrt {h}})v(t,y)dyd\xi \notag\\
		&=&\frac{2hi}{(2\pi h)^2} \int_{ \mathbb{R}^2}\int _{\mathbb{R}^2}e^{\frac{i}{h}(x-y)\cdot\xi}\gamma (\frac{\frac{x+y}{2}+F'(\xi)}{\sqrt{h}})v(t,y)dyd\xi \notag\\
		&&+\frac{1}{(2\pi h)^2} \sum_{k=1}^{2} \int_{\mathbb{R}^2} \int_{\mathbb{R}^2} e^{\frac{i}{h}(x-y)\cdot \xi} \gamma_k(\frac{\frac{x+y}{2}+F'(\xi)}{\sqrt{h}})\xi_k F_k''(\xi_k) i\sqrt {h} v(t,y)dyd\xi \notag\\
		&=& 2hiG_h^w(\gamma(\frac{x+F'(\xi)}{\sqrt{h}}))v+i\sqrt {h}\sum_{k=1}^{2}G_h^w(\gamma_k(\frac{x+F'(\xi)}{\sqrt{h}})\xi_k F_k''(\xi_k))v.\label{26w3}
	\end{eqnarray}
Substitution of  (\ref{26w3}) into (\ref{812s1}) yields
\begin{equation}
	[D_t,G_h^w(\Gamma)]v=i\sqrt {h} \sum_{k=1}^{2} G_h^w\left(\gamma_k(\frac{x+F'(\xi)}{\sqrt{h}})(\xi_k F_k''(\xi_k)-\frac{x_k+F_k'(\xi_k)}{2})\right)v.\label{26w4}
\end{equation}
	On the other hand, an application of  Proposition \ref{premainder} gives
	\begin{equation}
		[G^w_h(x\cdot \xi+F(\xi)),G^w_h(\Gamma)]v=i\sqrt {h} \sum_{k=1}^{2}G^w_h\left(\gamma_k(\frac{x+F'(\xi)}{\sqrt{h}})(\xi_kF_k''(\xi_k)-(x_k+F_k'(\xi_k)))\right)v
		+G_h^w(r)v,\label{6297}
	\end{equation}
	where   $r=:r^{(x\xi+F(\xi))\sharp \gamma(\frac{x+F'(\xi)}{\sqrt{h}})}-r^{\gamma(\frac{x+F'(\xi)}{\sqrt{h}})\sharp (x\xi+F(\xi))}$. Combining (\ref{26w4}) and (\ref{6297}), then using  Lemma \ref{sharp1}, we have for some symbols  $a_k(x,\xi),\ b_k(x,\xi)\in S_{\frac{1}{2}}(<\frac{x+F'(\xi)}{\sqrt{h}}>^{-2}),\ 1\le k\le2$
	\begin{eqnarray}
		&&[D_t-G_h^w(x\cdot\xi+F(\xi)),G_h^w(\Gamma)]v
		= \frac{i\sqrt {h}}{2}\sum_{k=1}^{2} G_h^w(\gamma_k(\frac{x+F'(\xi)}{\sqrt{h}})(x_k+F_k'(\xi_k)))v-G_h^w(r) v\notag\\
		&=& h\sum_{k=1}^{2}G_h^w(\Gamma^k(\frac{x+F'(\xi)}{\sqrt{h}}))v+h^2\sum_{k=1}^{2}\left(G_h^w(a_k)\circ \mathcal{L} _kv+G_h^w(b_k)v\right),\notag
	\end{eqnarray}
 where  $\Gamma^k (\xi)=:\frac{i\gamma_k(\xi)\xi_k}{2}\in C_0^\infty (\mathbb{R}^2)$, satisfying $\Gamma^k\equiv0$ in a neighborhood of zero.
	By Propositions \ref{pjl2}--\ref{pj2} and Lemma \ref{l4},  we find the estimates
	\begin{equation}
		\|[D_t-G_h^w(x\cdot\xi+F(\xi)),G_h^w(\Gamma)]v \|_{L_x^\infty }\lesssim h^{3/2}(\|\mathcal{L} ^2v\|_{L_x^2}+\|\mathcal{L} v\|_{L_x^2}+\|v\|_{L_x^2}),\notag
	\end{equation}
	\begin{equation}
	\|[D_t-G_h^w(x\cdot\xi+F(\xi)),G_h^w(\Gamma)]v \|_{L_x^2 }\lesssim h^{2}(\|\mathcal{L} ^2v\|_{L_x^2}+\|\mathcal{L} v\|_{L_x^2}+\|v\|_{L_x^2}),\notag
\end{equation}
which together with  (\ref{751}) yields  the desired estimates in Lemma \ref{l4.3}.
\end{proof}
\begin{lem}\label{l4.4}
	With the same assumptions as in Lemma \ref{lvc}, the following inequalities hold for all $t>1$
	\begin{eqnarray}
		\|G_h^w(x\cdot\xi+F(\xi)-w(x))v_{\Lambda }\|_{L_x^\infty }\lesssim t^{-3/2} (\|v\|_{L_x^2} +\|\mathcal{L} ^2v\|_{L_x^2}),\label{f1}
	\end{eqnarray}	
	\begin{eqnarray}
		\|G_h^w(x\cdot \xi+F(\xi)-w(x))v_{\Lambda }\|_{L_x^2 }\lesssim t^{-2} (\|v\|_{L_x^2} +\|\mathcal{L} ^2v\|_{L_x^2}).\label{f2}
	\end{eqnarray}
\end{lem}
\begin{proof}
	From Proposition \ref{premainder} and (\ref{w1})--(\ref{w2}), we have
	\begin{eqnarray}
		&&(x\cdot \xi+F(\xi)-w(x))\sharp \gamma(\frac{x+F'(\xi)}{\sqrt{h}})\notag\\
		&=& \gamma(\frac{x+F'(\xi)}{\sqrt{h}})\sharp  (x\cdot \xi+F(\xi)-w(x))\notag\\
&&+i\sqrt {h}\sum_{k=1}^{2} \gamma_k(\frac{x+F'(\xi)}{\sqrt{h}})\left[F_k''(\xi_k)(\xi_k-d\phi_k(x_k))-(x_k+F_k'(\xi_k))\right]\notag\\
		&&+(r^{(x\cdot\xi+F(\xi))\sharp \gamma(\frac{x+F'(\xi)}{\sqrt{h}})}-r^{\gamma(\frac{x+F'(\xi)}{\sqrt{h}})\sharp(x\cdot\xi+F(\xi))})-(r^{w(x)\sharp \gamma(\frac{x+F'(\xi)}{\sqrt{h}})}-r^{\gamma(\frac{x+F'(\xi)}{\sqrt{h}})\sharp w(x)}),\notag
	\end{eqnarray}
	where $\gamma_k(\xi)$ denotes $\partial_{\xi_k}\gamma(\xi)$. By  Lemmas \ref{sharp1}--\ref{sharp2}, there are symbols  $a_k(x,\xi),b_k(x,\xi)\in S_{\frac{1}{2}}(<\frac{x+F'(\xi)}{\sqrt{h}}>^{-2})$  such that
		\begin{eqnarray}
		&&(x\cdot \xi+F(\xi)-w(x))\sharp \gamma(\frac{x+F'(\xi)}{\sqrt{h}})\notag\\
		&=& \gamma(\frac{x+F'(\xi)}{\sqrt{h}})\sharp  (x\cdot \xi+F(\xi)-w(x))+h\sum_{k=1}^{2} \Gamma^k(\frac{x+F'(\xi)}{\sqrt{h}})E_k(x,\xi)\notag\\
		&&+\sum_{k=1}^{2}\left(ha_k(x,\xi)\sharp(x_k+F_k'(\xi_k))+h^2b_k(x,\xi)\right),\label{7211}
	\end{eqnarray}
where  $E_k(x,\xi)=:F_k''(\xi_k)\widetilde{e}_k(x,\xi)-1\in S_0(1)$ by (\ref{F}) and Lemma \ref{lone} and  $\Gamma^k(\xi)=:i\gamma_k(\xi)\xi_k\in C_0^\infty (\mathbb{R}^2)$, satisfying  $\Gamma^k\equiv0$ in a neighborhood of zero.

We now deal with  the first term on the right-hand side of (\ref{7211}).  	Since $\nabla _\xi (x\cdot\xi+F(\xi))|_{\xi=d\phi(x)}=0$, one would get
	\begin{eqnarray}
		&&x\cdot \xi+F(\xi)\notag\\
    &=& x \cdot d\phi(x)+F(d\phi(x))+\sum_{k=1}^{2}(\xi_k-d\phi_k(x_k))^2 \int_{0}^{1} F_k''(s\xi_k+(1-s)d\phi_k(x_k))(1-s) \mathrm{d}s \notag\\
		&=& w(x)+\sum_{k=1}^{2}\widetilde{b}_k(x,\xi) (x_k+F_k'(\xi_k))^2, \notag
	\end{eqnarray}
	where  $w(x)=x \cdot d\phi(x)+F(d\phi(x))$ and
$$
\widetilde{b}_k(x,\xi)=\widetilde{e}_k^2(x,\xi)\int_{0}^{1} F_k''(s\xi_k+(1-s)d\phi_k(x_k))(1-s) \mathrm{d}s, \ \ k=1,2.
$$
 Since  $\widetilde{b}_k(x,\xi)\in S_0(1)$ by (\ref{F}) and Lemma \ref{lone}, it follows from    (\ref{l1.3}) that
	\begin{equation}
		x\cdot\xi+F(\xi)-w(x)=\sum_{k=1}^{2} \left(\widetilde{b}_k(x,\xi)\sharp (x_k+F_k'(\xi_k))^2+hc_k(x,\xi)\sharp (x_k+F_k'(\xi_k))+h^2d_k(x,\xi)\right), \label{7133}
	\end{equation}
	for some symbols $c_k(x,\xi),d_k(x,\xi)\in S_{0}(1)$.

Substituting  (\ref{7133}) into (\ref{7211}), then taking the Weyl quantization, we get
	\begin{eqnarray}
		&&G_h^w(x\cdot\xi+F(\xi)-w(x))v_{\Lambda }\notag\\
		&=&h^2G_h^w(\gamma(\frac{x+F'(\xi)}{\sqrt{h}}))\circ \sum_{k=1}^{2}\left[G_h^w(\widetilde{b}_k)\circ \mathcal{L} ^2_kv+G_h^w(c_k)\circ\mathcal{L} _kv+G_h^w(d_k)v\right]\notag\\
		&&+h\sum_{k=1}^{2}G_h^w(\Gamma^k(\frac{x+F'(\xi)}{\sqrt{h}})E_k(x,\xi))v+h^2\sum_{k=1}^{2}\left[G_h^w(a_k)\circ \mathcal{L} _kv+G_h^w(b_k)v\right].\notag
	\end{eqnarray}
Therefore, using   Propositions \ref{pjl2}--\ref{pj2} and Lemma \ref{l4}, we estimate
\begin{equation}
	\|G_h^w(x\cdot\xi+F(\xi)-w(x))v_{\Lambda }\|_{L_x^\infty }\lesssim t^{-3/2}(\|v\|_{L_x^2}+\|\mathcal{L} v\|_{L_x^2}+\|\mathcal{L} ^2v\|_{L_x^2}),\notag
\end{equation}
\begin{equation}
	\|G_h^w(x\cdot\xi+F(\xi)-w(x))v_{\Lambda }\|_{L_x^2 }\lesssim t^{-2}(\|v\|_{L_x^2}+\|\mathcal{L} v\|_{L_x^2}+\|\mathcal{L} ^2v\|_{L_x^2}),\notag
\end{equation}
which together with  the Sobolev type inequality  (\ref{751}) yields the desired estimates in Lemma \ref{l4.4}.
\end{proof}

\begin{proof}
	[\textbf{Proof of the decay estimate (\ref{decay})}]
	It follows  from Lemmas \ref{l4.5}--\ref{l4.4} that there exists a constant $C_4>0$ such that
	\begin{equation}
		\|R(v)\|_{L_x^\infty } \le C_4 t^{-3/2} (1+\|v\|_{L_x^\infty }+\|v_{\Lambda }\|_{L_x^\infty })(\|v\|_{L_x^2}+\|\mathcal{L} ^2v\|_{L_x^2}),\label{r1}
	\end{equation}
	and
	\begin{equation}
		\|R(v)\|_{L_x^2 } \le C_4 t^{-2} (1+\|v\|_{L_x^\infty } +\|v_{\Lambda }\|_{L_x^\infty })(\|v\|_{L_x^2}+\|\mathcal{L} ^2v\|_{L_x^2}). \label{r2}
	\end{equation}
	Let
	\begin{equation}
		A=\max \left\{C_1,\ C_3,\ 24C_4,\ C_5 \right\},\label{A}
	\end{equation}
	\begin{equation}
		\varepsilon _0=\min \left\{\frac{1}{6A},\ \frac{1}{32A|\lambda |C_2} \right\}, \label{varepsilon }
	\end{equation}
	where $C_1,C_2,C_3,C_4,C_5$ are the constants in (\ref{7121}),  Lemma \ref{l2v}, Lemma \ref{lvc}, (\ref{r1}), (\ref{812s2}), respectively.

	In what follows, we assume that $0<\varepsilon <\varepsilon _0$ and $v$ satisfies the following bootstrap assumption on $t\in (1,T_1)$:
	\begin{equation}
		\|v(t,x)\|_{L^\infty _x} \le 4A\varepsilon. \label{boots}
	\end{equation}
	From   (\ref{7121}), (\ref{uv1}) and (\ref{A}),  we see that $T_1>1$.
	\begin{claim}\label{813z2} With the preceding notations,
		$	\|\mathcal{L} ^2v(t,\cdot)\|_{L^2}\le 2\varepsilon t^{8A|\lambda |C_2\varepsilon }.$
	\end{claim}
	\begin{proof}
		We notice the fundamental commutation property
		\begin{equation}
			[D_t-G_h^w(x\cdot\xi+F(\xi)),\mathcal{L} ^2]=0,\notag
		\end{equation}
		that follows by direct computation (one can also see that more easily going back to the non-semiclassical coordinates).
		Applying the operator $\mathcal{L} ^2$ to the equation  (\ref{TNLS}),  we get
		\begin{equation}
			(D_t-G_h^w(x\cdot\xi+F(\xi)))\mathcal{L} ^2v=\lambda t^{-1}\mathcal{L} ^2(\left|v\right| v). \label{791}
		\end{equation}
		Since $G_h^w(x\cdot\xi+F(\xi))$ is self-adjoint on $L^2$ by Proposition \ref{pz}, it follows from  Lemma \ref{zl}, Lemma \ref{l2v}, the Sobolev type inequality (\ref{lvalpha}) and the bootstrap assumption (\ref{boots}) that,
		\begin{eqnarray}
			\frac{d}{dt} \|\mathcal{L} ^2v\|_{L_x^2} &\le& \frac{2C_2 |\lambda | \|v\|_{L_x^\infty } }{t}  ( \|v\|_{L_x^2}+\|\mathcal{L} ^2v\|_{L_x^2}) \notag\notag\\
			&\le & \frac{8A|\lambda |C_2\varepsilon }{t} (\varepsilon +\|\mathcal{L} ^2v\|_{L_x^2}).\label{2143}
		\end{eqnarray}
By  (\ref{812z2}) and (\ref{812z1})
	\begin{equation}
		\|\mathcal{L} ^2v(1,\cdot)\|_{L^2}=\varepsilon \sum_{k=1}^{2}\|(x_k+F_k'(D))u_0\|_{L^2}\le \varepsilon ,\notag
	\end{equation}
	we get after integrating the  inequality (\ref{2143}) from  $1$ to  $t$
	\begin{equation}
		\varepsilon +\|\mathcal{L} ^2v(t,\cdot)\|_{L^2}\le 2\varepsilon +\int_{1}^{t}\frac{8A|\lambda |C_2\varepsilon }{s}(\varepsilon +\|\mathcal{L} ^2v(s,\cdot)\|_{L^2}) \mathrm{d}s.\notag
	\end{equation}
		From  Grownwall's inequality,  we have the desired estimate in Claim \ref{813z2}.
	\end{proof}
	Lemma \ref{lvc} and Claim \ref{813z2} imply
	\begin{equation}
		\|v_{\Lambda ^c}(t,x)\|_{L^\infty _x}\le C_3t^{-1/2}(\varepsilon +2\varepsilon t^{8A|\lambda |C_2\varepsilon }) \le 4C_3\varepsilon t^{-1/2+8A|\lambda |C_2\varepsilon }.\label{896}
	\end{equation}
	This  together with  the bootstrap assumption (\ref{boots}) and (\ref{A})  yields
	\begin{equation}
		\|v_{\Lambda }(t,x)\|_{L^\infty _x} \le \|v(t,x)\|_{L^\infty _x}+\|v_{\Lambda^c }(t,x)\|_{L^\infty _x}\le  4A\varepsilon +4C_3\varepsilon  \le 8A\varepsilon .\label{897}
	\end{equation}
Combining   (\ref{boots}), (\ref{897}) and Claim \ref{813z2}, we can upgrade the estimates  of  $R(v)$ in   (\ref{r1})--(\ref{r2}) to
	\begin{equation}
		\|R(v)\|_{L_x^\infty }\le 6C_4\varepsilon t^{-3/2+8A|\lambda |C_2\varepsilon },\qquad \|R(v)\|_{L_x^2}\le 6C_4\varepsilon t^{-2+8A|\lambda |C_2\varepsilon }.\label{898}
	\end{equation}
	\begin{claim}
		\label{813z3} With the preceding notations,
		$	\|v_{\Lambda }(t,\cdot)\|_{L^\infty }\le 2A\varepsilon. $
	\end{claim}
	\begin{proof}
		Multiplying  the equation  (\ref{TNLS}) by $\overline{v_{\Lambda }}$,  then taking Imaginary part,  we get
		\begin{eqnarray}
			\partial_t |v_{\Lambda }|^2=-\text{Im} \left(R(v)\overline{v_{\Lambda }}\right),\notag
		\end{eqnarray}
		which together with  (\ref{898})  and  (\ref{A})--(\ref{varepsilon }) implies
		\begin{eqnarray}
			\|v_{\Lambda }(t,x)\|_{L^\infty _x}&\le & \|v_{\Lambda }(1,x)\|_{L^\infty _x} +\int_1^t\|R(v)\|_{L^\infty _x} ds \notag\\
			&\le& C_5\varepsilon +6C_4\varepsilon \int_1^t s^{-3/2+8A|\lambda |C_2\varepsilon }ds \notag\\
			&\le& C_5\varepsilon +\frac{12C_4\varepsilon }{1-16A|\lambda |C_2\varepsilon }\le 2A\varepsilon ,\label{812z3}
		\end{eqnarray}
		where the second inequality holds since by  Proposition \ref{pj1}, there exists a absolute constant  $C_5>0$
		\begin{equation}
			\|v_{\Lambda }(1,x)\|_{L^\infty _x} =\|G_h^w(\gamma(\frac{x+F'(\xi)}{\sqrt{h}}))v|_{h=1}\|_{L^\infty _x}\le C_5\|v\|_{L_x^2}\le C_5 \varepsilon. \label{812s2}
		\end{equation}
	\end{proof}
	The estimates in   (\ref{896})   and Claim \ref{813z3} imply
	\begin{equation}
		\|v(t,x)\|_{L^\infty_x }\le 	\|v_{\Lambda }(t,x)\|_{L^\infty_x }+	\|v_{\Lambda ^c}(t,x)\|_{L^\infty_x } \le  2A\varepsilon +4C_3\varepsilon \le 3A\varepsilon. \label{812z4}
	\end{equation}
A standard continuation argument then yields that  $T_1=\infty $ and the estimates in Claim \ref{813z2}, Claim \ref{813z3} and (\ref{812z4}) hold for all  $t>1$. Estimate (\ref{812z4}) together with (\ref{uv1}) yields the desired time decay estimate (\ref{decay}), from which  Theorem \ref{T1} follows.
\end{proof}

\section{Proof of Theorems \ref{T2}--\ref{T4}}\label{s4}
In this section, we prove Theorems \ref{T2}--\ref{T4}  by deducing the
 long-time behavior of solutions from the associated ODE dynamics. It is divided into three subsections.
\subsection{Proof of Theorem \ref{T2}}
The proof is inspired by  Section 4 of Shimomura \cite{CPDE}. By the definition of $\Phi(t,x)$ in  (\ref{Phi}) and Claim 3.2, we have that
\begin{equation}
	\|\Phi(t,x)\|_{L^\infty _x}\le \int _1^t s^{-1} \|v_{\Lambda }(t,x)\|_{L^\infty _x} ds \le 2A\varepsilon \log t,\qquad t>1. \label{824w2}
\end{equation}
Let
\begin{equation}
	z(t,x)=v_{\Lambda }(t,x)e^{-i(w(x)t+\lambda \Phi(t,x))}, \qquad t>1.\label{824s1}
\end{equation}
From the equation (\ref{4.34}) and (\ref{824s1}), we have that
\begin{equation}
	\partial_t  z(t,x)=iR(v) e^{-i(w(x)t+\lambda \Phi(t,x))};\notag
\end{equation}
so that for all $t_2>t_1>1$
\begin{equation}
	z(t_2,x)-z(t_1,x)=i\int _{t_1}^{t_2}R(v) e^{-i(w(x)s+\lambda \Phi(s,x))}ds.\notag
\end{equation}
Estimates (\ref{898}) and (\ref{824w2}) imply
\begin{eqnarray}
	\|z(t_2,x)-z(t_1,x)\|_{L^\infty _x}&\lesssim &  \varepsilon \int_{t_1}^{t_2} \|R(v)\|_{L^\infty _x} e^{\text{Im} \lambda  \|\Phi(s,x)\|_{L^\infty _x}}ds \notag\\
	&\lesssim &  \varepsilon  \int_{t_1}^{t_2} s^{-3/2+C_6\varepsilon +2\text{Im} \lambda A\varepsilon } ds \notag\\
	&\lesssim &  \varepsilon  t_1^{-1/2+C_6\varepsilon +2\text{Im} \lambda A\varepsilon } ,\notag
\end{eqnarray}
where  $C_6=8A|\lambda |C_2 $. A similar argument shows that
\begin{equation}
	\|z(t_2,x)-z(t_1,x)\|_{ L^2_x}\lesssim \varepsilon  t_1^{-1+C_6\varepsilon +2\text{Im} \lambda A\varepsilon },\qquad \text{for all } t_2>t_1>1.\notag
\end{equation}
Therefore,  there exists $z_+(x)\in L^\infty _x\cap L^2_x$ such that
\begin{equation}
	\|z(t,x)-z_+(x)\|_{L^\infty _x}\lesssim  \varepsilon t^{-1/2+C_6\varepsilon +2\text{Im} \lambda  A\varepsilon },\ 	\|z(t,x)-z_+(x)\|_{L^2 _x}\lesssim \varepsilon  t^{-1+C_6\varepsilon +2\text{Im} \lambda  A\varepsilon },\label{824z1}
\end{equation}
from which  Theorem \ref{T2} follows.
\subsection{Proof of Theorem \ref{T3}}
Since   $\text{Im} \lambda =0$ and  $\Phi(t,x)$ is a real-valued function, it follows from (\ref{824s1}) that  $|v_\Lambda (t,x)|=|z(t,x)|$.
Recalling  the definitions of $\Phi(t,x)$ and $\phi_+(x)$ in (\ref{Phi}), (\ref{phi}), it follows that
\begin{equation}
	\phi_+(x)+|z_+(x)|\log t-\Phi(t,x)=\int_t^\infty s^{-1} (|z(s,x)|-|z_+(x)|)ds.\label{271}
\end{equation}
Applying  (\ref{824z1}) to (\ref{271}), we find the estimate
\begin{equation}
	\|\phi_+(x)+|z_+(x)|\log t -\Phi(t,x)\|_{L^\infty _x} \lesssim t^{-1/2+C_6\varepsilon  }.\label{824w3}
\end{equation}
Estimates  (\ref{824z1}) and (\ref{824w3}) imply
\begin{eqnarray}
	&&\|z_+(x)e^{i(w(x)t+\lambda \phi_+(x)+\lambda |z_+(x)|\log t )}-v_{\Lambda }\|_{L^\infty _x} \notag\\
	&\lesssim &   \|z_+(x)e^{i(w(x)t+\lambda \phi_+(x)+\lambda |z_+(x)|\log t )} -z_+(x) e^{i(w(x)t+\lambda \Phi(t,x))}\|_{L^\infty _x}\notag\\
	&&+ \|z_+(x) e^{i(w(x)t+\lambda \Phi(t,x))}- v_{\Lambda } (t,x)\|_{L^\infty _x} \notag\\
	&\lesssim &   \|	\phi_+(x)+|z_+(x)|\log t-\Phi(t,x)\|_{L^\infty _x}+\|z_+(x)-z(t,x)\|_{L^\infty _x}\notag\\
	&\lesssim &    t^{-1/2+C_6\varepsilon  }. \label{214x1}
\end{eqnarray}
On the other hand,  from   Lemma \ref{lvc} and Claim \ref{813z2}
\begin{equation}
	\|v_{\Lambda ^c}(t,x)\|_{L^\infty _x}\lesssim t^{-1/2+C_6\varepsilon },\qquad  \|v_{\Lambda ^c}(t,x)\|_{L^2_x}\lesssim  t^{-1+C_6\varepsilon },\qquad  t>1,\label{2101}
\end{equation}
we obtain, by substituting (\ref{2101}) into (\ref{214x1})
\begin{equation}
	\|v(t,x)-z_+(x)e^{i(w(x)t+\lambda \phi_+(x)+\lambda |z_+(x)|\log t )}\|_{L^\infty _x}\lesssim  t^{-1/2+C_6\varepsilon },\qquad t>1.\label{272}
\end{equation}
In the same manner, we get
\begin{equation}
		\|v(t,x)-z_+(x)e^{i(w(x)t+\lambda \phi_+(x)+\lambda |z_+(x)|\log t )}\|_{L^2 _x}\lesssim  t^{-1+C_6\varepsilon },\qquad t>1.\label{273}
\end{equation}
The asymptotic formula (\ref{T3a}) then follows from the estimates (\ref{272})--(\ref{273}) and (\ref{uv1}).

It remains to  derive the modified linear scattering formula (\ref{T3b}). From the  asymptotic formula (\ref{T3a}), we have
\begin{eqnarray}
	&&e^{-iF(D)t}e^{-i\lambda  \left(\phi_+(\frac{x}{t})+|z_+(\frac{x}{t})|\log t\right)  }u(t,x)
	=e^{-iF(D)t}\frac{1}{ t}z_+(\frac{x}{t})e^{itw(\frac{x}{t})}+O_{L^2}(t^{-1+C\varepsilon })\notag\\
	&&= \frac{1}{(2\pi)^2} \int_{\mathbb{R}^2}\int_{\mathbb{R}^2 } e^{i(x-y)\cdot \xi}e^{-iF(\xi)t}\frac{1}{ t}z_+(\frac{y}{t})e^{itw(\frac{y}{t})}dyd\xi+O_{L^2}(t^{-1+C\varepsilon }).\label{1081}
\end{eqnarray}
as $t\rightarrow \infty $.  Making a  change of variables  then  using (\ref{15100}), we obtain
\begin{eqnarray}
	&&e^{-iF(D)t}e^{-i\lambda  \left(\phi_+(\frac{x}{t})+|z_+(\frac{x}{t})|\log t\right)  }u(t,x)+O_{L^2}(t^{-1+C\varepsilon })\notag\\
	&=& \frac{ t}{(2\pi)^2 }\int \int e^{ix\cdot \xi-it\int_0^1 \sum_{k=1}^{2}F_k''(\tau\xi_k+(1-\tau)d\phi_k(y_k))(1-\tau)d\tau (\xi_k-d\phi_k(y_k))^2}z_+(y)dyd\xi\notag\\
	&=&  \frac{1}{(2\pi)^2} \int \int  e^{ix\cdot(\frac{\eta}{\sqrt {t}}+d\phi(y))-i\sum_{k=1}^{2}\eta_k^2\int_0^1 F_k''(d\phi_k(y_k)+\tau \frac{\eta_k}{\sqrt {t}})(1-\tau)d\tau}z_+(y)dyd\eta.\notag	
\end{eqnarray}
Letting  $t\rightarrow \infty $ and using the dominated convergence theorem, we get
\begin{eqnarray}
	&&e^{-iF(D)t}e^{-i\lambda  \left(\phi_+(\frac{x}{t})+|z_+(\frac{x}{t})|\log t\right)  }u(t,x)\notag\\
	&\xlongequal{L^2}&  \frac{1}{(2\pi)^2} \int \int  e^{ix\cdot d\phi(y)-\frac{i}{2}\sum_{k=1}^{2}\eta_k^2F_k''(d\phi_k(y_k))}z_+(y)dyd\eta\notag\\
	&\xlongequal{L^2}& -\frac{i}{2\pi} \int  e^{ix\cdot d\phi(y)}\frac{1}{\sqrt {\prod_{k=1}^2F_k''(\phi_k(y_k) ) }}z_+(y)dy=u_+ (x) ,\notag
\end{eqnarray}
from which the  modified scattering formula (\ref{T3b}) follows.
\subsection{Proof of Theorem \ref{T4}}
We start by  deriving the  asymptotic formula  (\ref{825w2}) for $\Phi(t,x)$, since (\ref{824w3}) does not hold in the case $\text{Im} \lambda >0$.  Note that by the definition of  $\Phi(t,x)$ in (\ref{Phi}) and (\ref{824s1})
\begin{equation}
	\partial_t  \Phi(t,x)=t^{-1}|v_{\Lambda }(t,x)|=t^{-1} |z(t,x)|e^{-\text{Im} \lambda \Phi(t,x)}, \qquad  t>1;\notag
\end{equation}
so that
\begin{equation}
	\partial_t  e^{\text{Im} \lambda \Phi(t,x)}=\text{Im} \lambda t^{-1} |z(t,x)|,\qquad t>1.\notag
\end{equation}
Integrating the above equation form $1$ to $t$, we get
\begin{equation}
	e^{\text{Im} \lambda \Phi(t,x)}=1+\text{Im} \lambda \int_{1}^{t} s^{-1} |z(s,x)| \mathrm{d}s,\qquad t>1.\notag
\end{equation}
Recalling  the definition of $\psi_+(x)$ in (\ref{psi}), it follows that
\begin{equation}
	e^{\text{Im} \lambda \Phi(t,x)}-(1+\text{Im} \lambda |z_+(x)|\log t)-\psi_+(x)=-\text{Im} \lambda \int_t^\infty s^{-1}(|z(s,x)|-|z_+(x)|)ds. \label{825w1}
\end{equation}
Applying (\ref{824z1}), we find the estimate
\begin{equation}
	\|	e^{\text{Im} \lambda \Phi(t,x)}-(1+\text{Im} \lambda |z_+(x)|\log t)-\psi_+(x)\|_{L^\infty _x}\lesssim \varepsilon t^{-1/2+C_6\varepsilon +2\text{Im} \lambda A\varepsilon }.\label{825w2}
\end{equation}
In particular, we have that \begin{equation}
	1+\text{Im} \lambda |z_+(x)|\log t+\psi_+(x)\ge \frac{1}{2}\label{351}
\end{equation}
provided that  $\varepsilon >0$ is sufficiently small.

We now prove the asymptotic formula (\ref{T4a}). By triangle inequality
\begin{eqnarray}\label{3304}
	&&\|e^{i(w(x)t+\lambda S(t,x))}z_+(x)-v_{\Lambda }(t,x)\|_{L_x^\infty }\notag\\
	&\le& \|e^{iw(x)t}(e^{i\lambda S(t,x)}-e^{i\lambda \Phi(t,x)})z_+\|_{L_x^\infty }+\|e^{iw(x)t}e^{i\lambda \Phi(t,x)}z_+-v_{\Lambda }(t,x)\|_{L_x^\infty }\notag\\
	&\lesssim &\|e^{i\text{Re}\lambda  S(t,x)}(e^{-\text{Im} \lambda S(t,x)}-e^{-\text{Im} \lambda \Phi(t,x)})\|_{L_x^\infty }+\|(e^{i\text{Re}\lambda S(t,x)}-e^{i\text{Re}\lambda \Phi(t,x)})e^{-\text{Im} \lambda \Phi(t,x)}\|_{L_x^\infty} \notag\\
	&&+\|z_+(x)-z(t,x)\|_{L_x^\infty },
\end{eqnarray}
where we used $|e^{iw(x)t+i\lambda \Phi(t,x)}|=e^{-\text{Im} \lambda \Phi(t,x)}\le1$.
By   the definition of $S(t,x)$ in (\ref{S}), (\ref{351})  and the estimate (\ref{825w2}), we have
\begin{eqnarray}\label{3225}
	&&\|e^{i\text{Re}\lambda S(t,x)}(e^{-\text{Im} \lambda S(t,x)}-e^{-\text{Im} \lambda \Phi(t,x)})\|_{L_x^\infty }\notag\\
	&\le& \|(1+\text{Im} \lambda |z_+(x)|\log t+\psi_+(x))^{-1 }-e^{-\text{Im} \lambda \Phi(t,x)}\|_{L_x^\infty}\notag\\
	&\lesssim &\|e^{\text{Im} \lambda \Phi(t,x)}-(1+\text{Im} \lambda |z_+(x)|\log t+\psi_+(x))\|  _{L_x^\infty}\notag\\
	&\lesssim & t^{-1/2+C_6\varepsilon +2\text{Im} \lambda A\varepsilon  }.
\end{eqnarray}
 A similar argument shows that
\begin{eqnarray}\label{3305}
	&&\|(e^{i\text{Re}\lambda S(t,x)}-e^{i\text{Re}\lambda \Phi(t,x)})e^{-\text{Im} \lambda \Phi(t,x)}\|_{L_x^\infty}\notag\\
	&\lesssim &  \|S(t,x)-\Phi(t,x)\|_{L_x^\infty }\lesssim   \|e^{\text{Im} \lambda S(t,x)}-e^{\text{Im} \lambda \Phi(t,x)}\|_{L_x^\infty}\notag\\
	&\lesssim & \|(1+\text{Im} \lambda |z_+(x)|\log t+\psi_+(x))-e^{\text{Im} \lambda \Phi(t,x)}\|  _{L_x^\infty}
	\lesssim   t^{-1/2+C_6\varepsilon +2\text{Im} \lambda A\varepsilon  }.
\end{eqnarray}
Substituting the estimates (\ref{3225}), (\ref{3305}) and (\ref{824z1}) into (\ref{3304}), and  using  (\ref{2101}), we get
\begin{equation}
	\label{3306}
	\|e^{i(w(x)t+\lambda S(t,x))}z_+(x)-v(t,x)\|_{L_x^\infty }\lesssim  t^{-1/2+C_6\varepsilon +2\text{Im} \lambda A\varepsilon  },\qquad t>1.
\end{equation}
In the same manner, we get
\begin{equation}
	\label{33061}
	\|e^{i(w(x)t+\lambda S(t,x))}z_+(x)-v(t,x)\|_{L_x^2 }\lesssim  t^{-1+C_6\varepsilon +2\text{Im} \lambda A\varepsilon  },\qquad t>1.
\end{equation}
 The asymptotic formula (\ref{T4a}) then follows from  the estimates (\ref{3306})--(\ref{33061}) and (\ref{uv1}).

Using the same method as that used to derive (\ref{T3b}), we obtain the modified linear scattering formula (\ref{T4b}) easily and omit the details.

It remains to prove  the limit (\ref{T4c}).  By (\ref{uv1}), it is equivalent to proving that
\begin{equation}
	\lim_{t\rightarrow \infty }\log t\|v(t,x)\|_{L^\infty _x}=\frac{1}{\text{Im} \lambda }.\label{2102}
\end{equation} By the definition of  $\psi_+(x)$ in (\ref{psi}) and the estimate (\ref{824z1}), we have that
\begin{equation}
	\|\psi_+(x)\|_{L^\infty _x}\lesssim \int_{1}^{\infty }s^{-1}\|z(s,x)-z_+(x)\|_{L^\infty _x} \mathrm{d}s\lesssim \varepsilon \int_{1}^{\infty }s^{-2+C_6\varepsilon +2\text{Im} \lambda A\varepsilon } \mathrm{d}s\lesssim \varepsilon .\notag
\end{equation}
Thus we have  $1+\psi_+(x)\ge0$, provided that  $\varepsilon >0$ is sufficient small. It then  follows from the asymptotic formula (\ref{3306}) and (\ref{T4d}) that
\begin{equation}
	\log t |v(t,x)|\le \frac{|z_+(x)|\log t}{1+\text{Im} \lambda |z_+(x)|\log t+\psi_+(x)}+O(t^{-1/2+C\varepsilon }\log t)\le \frac{1}{\text{Im} \lambda }+O(t^{-1/2+C\varepsilon }\log t),\notag
\end{equation}
which implies
\begin{equation}
	\limsup_{t\rightarrow \infty } \log t \|v(t,x)\|_{L^\infty _x} \le \frac{1}{\text{Im} \lambda }. \notag
\end{equation}
Therefore, for  (\ref{2102}) it suffices to prove that
\begin{equation}
	\liminf_{t\rightarrow \infty }\log t\|v(t,x)\|_{L^\infty _x}\ge \frac{1}{\text{Im} \lambda }.\label{827x3}
\end{equation}
Assume for a while that we have proved
\begin{claim}
	\label{827x1}
	If the limit function $z_+(x)$ in (\ref{824z1}) satisfies $z_+(x)=0$ for a.e. $x\in \mathbb{R}^2$, then we must have $u_0=0$.
\end{claim}
Since $u_0\neq0$, there exists $x_0\in \mathbb{R}^2$ such that $z_+(x_0)\neq0$. Therefore, using   (\ref{3306}) and (\ref{T4d}), we estimate
\begin{equation}
	\log t\|v(t,x)\|_{L^\infty _x}\ge \frac{\log t|z_+(x_0)|}{ 1+\text{Im} \lambda |z_+(x_0)|\log t+\psi_+(x_0)}+O(t^{-1/2+C\varepsilon }\log t).\notag
\end{equation}
Letting  $t\rightarrow \infty $, we get  the limit (\ref{827x3}), from which the desired  limit (\ref{T4c}) follows.

\begin{proof}[\textbf{Proof of Claim \ref{827x1}}]
	Since $z_+=0$, it follows from (\ref{824s1}) and (\ref{824z1}) that
	\begin{equation}
		\|v_{\Lambda }(t,x)\|_{L^\infty _x} \lesssim  \|z(t,x)\|_{L^\infty _x}\lesssim t^{-1/2+C\varepsilon },\qquad \forall t>1\notag
	\end{equation}
	which together with (\ref{2101}) and (\ref{uv1}) implies
	\begin{equation}
		\|u(t,x)\|_{L^\infty _x}\lesssim t^{-3/2+C\varepsilon },\qquad \forall t>1.\label{827x2}
	\end{equation}
	On the other hand,  since $z_+=0$ implies  $u_+=0$ (see (\ref{274})), it follows from the equation (\ref{NLS}), the asymptotic formula (\ref{T4b}) and Duhamel's formula that
	\begin{equation}
		u(t,x)=\lambda \int_t^\infty e^{iF(D)(t-s)}(|u|u)(s)ds.\notag
	\end{equation}
	Using successively Strichartz's estimate, H\"older's inequality and (\ref{827x2}) to get
	\begin{eqnarray}
		\|u(t,x)\|_{L_t^\infty ([T,\infty ),L^2_x)}&\lesssim& \int_T^\infty \|u(t,x)\|_{L_t^\infty ([T,\infty ),L^2_x)}\|u(s,x)\|_{L_x^\infty }ds\notag\\
		&\lesssim &  	\|u(t,x)\|_{L_t^\infty ([T,\infty ),L^2_x)}T^{-1/2+C\varepsilon }.\notag
	\end{eqnarray}
	Choosing $T>1$ sufficiently large, we deduce that $	\|u(t,x)\|_{L_t^\infty ([T,\infty ),L^2_x)}=0$, which together with the uniqueness of solutions implies $u\equiv 0$. This finishes the proof of Claim \ref{827x1}.
\end{proof}
\section*{Appendix}
\setcounter{equation}{0}
\setcounter{subsection}{0}
\renewcommand{\theequation}{A.\arabic{equation}}
\renewcommand{\thesubsection}{A.\arabic{subsection}}
This appendix is devoted to the proof of Lemma \ref{l2v}.
\begin{proof}[\textbf{Proof of Lemma \ref{l2v}}]
We start by recalling that (see Lemma \ref{lone})
\begin{equation}
	e_k(x,\xi)=\frac{x_k+F_k'(\xi_k)}{\xi_k-d\phi_k(x_k)}\in S_0(1),\qquad \widetilde{e}_k(x,\xi)=\frac{\xi_k-d\phi_k(x_k)}{x_k+F'_k(\xi_k)}\in S_0(1).\notag
\end{equation}
 Using   (\ref{l1.3}) and (\ref{l1.4}), we can write, for some  symbols $c_{kj},\ d_{kj}\in S_0(1), 1\le k,j\le2$
	\begin{eqnarray}
		&&(x_k+F_k'(\xi_k))^2=e_k^2(x,\xi)(\xi_k-d\phi_k(x_k))^2\notag\\
		&=&e_k^2(x,\xi)\sharp (\xi_k-d\phi_k(x_k))^2+hc_{k1}(x,\xi)\sharp (x_k+F_k'(\xi_k))+h^2d_{k1}(x,\xi) ,\label{6283}
	\end{eqnarray}
	\begin{eqnarray}
		&&(\xi_k-d\phi_k(x_k))^2=\widetilde{e}_k^2(x,\xi)(x_k+F_k'(\xi_k))^2\notag\\
		&=&\widetilde{e}_k^2(x,\xi)\sharp (x_k+F_k'(\xi_k))^2+hc_{k2}(x,\xi)\sharp (\xi_k-d\phi_k(x_k))+h^2d_{k2}(x,\xi).\label{6284}
	\end{eqnarray}
	Moreover, using (\ref{l1.2}),  we can write
	\begin{eqnarray}
		x_k+F_k'(\xi_k)=e_k(x,\xi)(\xi_k-d\phi_k(x_k))=e_k(x,\xi)\sharp (\xi_k-d\phi_k(x_k))+hb_k(x,\xi),\ k=1,2,\label{741}
	\end{eqnarray}
	for  some symbols $b_k(x,\xi)\in S_0(1)$, $k=1,2$.
	Substituting (\ref{741}) into (\ref{6283}), then taking the Weyl quantization, we get
	\begin{eqnarray}
		\mathcal{L} _k^2&=&\frac{1}{h^2} G_h^w(e_k^2)\circ G_h^w((\xi_k-d\phi_k(x_k))^2) +G_h^w(c_{k1})\circ G_h^w(e_k)\circ \frac{1}{h}G_h^w(\xi_k-d\phi_k(x_k))\notag\\
		&&+G_h^w(c_{k1})\circ G_h^w(b_k)+G_h^w(d_{k1}),\ k=1,2.\label{6285}
	\end{eqnarray}
Therefore, using  Proposition \ref{pjl2} we estimate
\begin{eqnarray}
		\|\mathcal{L} ^2_k(|v|v)\|_{L^2}&\lesssim&    \|\frac{1}{h^2}G_h^w((\xi_k-d\phi_k(x_k))^2)(|v|v)\|_{L^2}\notag\\
		&&+\|\frac{1}{h}G_h^w(\xi_k-d\phi_k(x_k))(|v|v)\|_{L^2}+\|v\|_{L^\infty }\|v\|_{L^2}.\label{214x2}
\end{eqnarray}

Next, we consider the estimate of the right-hand side of (\ref{214x2}).  Since
	\begin{eqnarray}
		\partial_{x_kx_k}(|v| v)&=&\frac{3}{2}|v| \partial_{x_kx_k}v +\frac{1 }{2}|v|^{-1}v^2 \overline{\partial_{x_kx_k}v}+\frac{3}{2}|v|^{-1}\partial_{x_k} v\text{Re}(\partial_{x_k} v\overline{v})\notag\\
		&&-\frac{1}{2}|v|^{-3}v^2 \overline{\partial_{x_k} v}\text{Re}(\partial_{x_k} v\overline{v})+ |v|^{-1}v |\partial_{x_k} v|^2,\qquad k=1,2.\label{891}
	\end{eqnarray}
	we have, by direct calculation
	\begin{eqnarray}
		&&	G_h^w((\xi_k-d\phi_k(x_k))^2)(|v| v)\notag\\
		&=& [-h^2\partial_{x_k} ^2+((d\phi_k (x_k))^2+hid^2\phi_k (x_k))+2hi d\phi_k (x_k) \partial_{x_k} ](|v| v)\notag\\
		&=& -\frac{3}{2}h^2|v| \partial_{x_kx_k}v -\frac{1 }{2}h^2|v|^{-1}v^2\overline{\partial_{x_kx_k}v}-\frac{3}{2}h^2|v|^{-1}\partial_{x_k} v \text{Re}(\partial_{x_k} v\overline{v})\notag\\
		&&+\frac{1}{2}h^2|v|^{-3}v^2\overline{\partial_{x_k} v}\text{Re} (\partial_{x_k} v\overline{v})- h^2|v|^{-1}v |\partial_{x_k} v|^2+\left((d\phi_k(x_k))^2+hid^2\phi_k (x_k)\right)(|v| v)\notag\\
		&&+2hi d\phi_k (x_k)(\frac{3}{2}|v| \partial_{x_k} v+\frac{1 }{2}|v|^{-1}v^2\overline{\partial_{x_k} v})\notag\\
		&=& \frac{3}{2}|v| (-h^2\partial_{x_k} ^2+((d\phi_k (x_k))^2+hid^2\phi_k (x_k))+2hi d\phi_k (x_k) \partial_{x_k} )v\notag\\
		&&+\frac{1 }{2}|v|^{-1}v^2 \overline{(-h^2\partial_{x_k} ^2+((d\phi_k (x_k))^2+hid^2\phi_k (x_k))+2hi d\phi_k (x_k) \partial_{x_k} )v}\notag\\
		&&-\frac{3}{2}|v|^{-1}(h\partial_{x_k} v-id\phi_k(x_k)v)\text{Re}( (h\partial_{x_k} v-id\phi_k(x_k)v)\overline{v})\notag\\
		&&+\frac{1}{2}|v|^{-3}v^2 \overline{h\partial_{x_k} v-id\phi_k(x_k) v}\text{Re} ((h\partial_{x_k} v-id\phi_k(x_k) v)\overline{v})\notag\\
		&&- |v|^{-1}v |h\partial_{x_k} v-id\phi_k(x_k) v|^2,\ k=1,2. \notag
	\end{eqnarray}
	Substituting  $h\partial_{x_k} -id\phi_k(x_k)=iG_h^w(\xi_k-d\phi_k(x_k))$ into the above expression, we get
	\begin{eqnarray}
		&&G_h^w((\xi_k-d\phi_k(x_k))^2)(|v| v)\notag\\
		&=& \frac{3}{2}|v| G_h^w((\xi_k-d\phi_k(x_k) )^2)v+\frac{1 }{2}|v|^{-1}v^2 \overline{G_h^w((\xi_k-d\phi_k(x_k) )^2)v}\notag\\
		&&+\frac{3}{2}|v|^{-1}iG_h^w(\xi_k-d\phi_k(x_k) )v \text{Im} (G_h^w(\xi_k-d\phi_k(x_k) )(v)\overline{v})\notag\\
		&&-\frac{1}{2}|v|^{-3}v^2 \overline{iG_h^w(\xi_k-d\phi_k(x_k) )v} \text{Im}  (G_h^w((\xi_k-d\phi_k(x_k) )(v)\overline{v})\notag\\
		&&- |v|^{-1}v |G_h^w(\xi_k-d\phi_k(x_k) )v|^2,\ k=1,2.\label{752}
	\end{eqnarray}
By Proposition \ref{pjl2}, we have,  for $k=1,2$
\begin{eqnarray}
		&&\|\frac{1}{h^2}G_h^w((\xi_k-d\phi_k(x_k))^2)(|v|v)\|_{L_x^2}\notag\\
		&\lesssim& \|v\|_{L_x^\infty } \|\frac{1}{h^2}G_h^w((\xi_k-d\phi_k(x_k))^2)v\|_{L_x^2}+ \|\frac{1}{h}G_h^w(\xi_k-d\phi_k(x_k))v\|_{L_x^4}^2. \notag
\end{eqnarray}
Substitution of the above inequality  into (\ref{214x2}) gives,  for $k=1,2$
	\begin{eqnarray}
		&&\|\mathcal{L} _k^2(|v| v)\|_{L_x^2}+\|\frac{1}{h^2}G_h^w((\xi_k-d\phi_k(x_k))^2)(|v|v)\|_{L_x^2}\notag\\
		&\lesssim& \|v\|_{L_x^\infty } (\|v\|_{L_x^2}+\|\frac{1}{h^2}G_h^w((\xi_k-d\phi_k(x_k))^2)v\|_{L_x^2}) \notag\\
		&&+ \|\frac{1}{h}G_h^w(\xi_k-d\phi_k(x_k))v\|_{L_x^4}^2+\|\frac{1}{h}G_h^w(\xi_k-d\phi_k(x_k))(|v|v)\|_{L_x^2}. \label{6294}
	\end{eqnarray}
	We now  estimate  the right-hand side of (\ref{6294}). Notice that
	\begin{equation}
		\partial_{x_k} e^{-it\phi_k(x_k)}v=ie^{-it\phi_k(x_k)}\frac{1}{h}G_h^w(\xi_k-d\phi_k(x_k))v,\label{6292}
	\end{equation}
	\begin{equation}
		\partial_{x_{kk}} e^{-it\phi_k(x_k)}v=-e^{-it\phi_k(x_k)}\frac{1}{h^2}G_h^w((\xi_k-d\phi_k(x_k))^2)v;\label{813z1}
	\end{equation}
	so that by   Gagliardo-Nirenberg's inequality
	\begin{eqnarray}
		&&\|\frac{1}{h}G_h^w(\xi_k-d\phi_k(x_k))v\|_{L_x^4}^2=\|\partial_{x_k} e^{-it\phi_k(x_k)}v\|_{L_x^4}^2 \notag\\
		&\lesssim &   \|e^{-it\phi_k(x_k)}v\|_{L_x^\infty } \|\partial_{x_k x_k} e^{-it\phi_k(x_k)}v  \|_{L_x^2}
		\lesssim    \|v\|_{L_x^\infty } \|\frac{1}{h^2}G_h^w((\xi_k-d\phi_k(x_k))^2)v\|_{L_x^2}. \label{6293}
	\end{eqnarray}
	By a similar argument and Young's inequality,  we find the estimate
	\begin{eqnarray}
	&&	\|\frac{1}{h}G_h^w(\xi_k-d\phi_k(x_k))(|v|v)\|_{L_x^2}\notag\\
	&\lesssim &   \|e^{-it\phi_k(x_k)}(|v|v)\|_{L^2 } ^{\frac{1}{2}}\|\partial_{x_k x_k} e^{-it\phi_k(x_k)}(|v|v)  \|_{L_x^2}^{\frac{1}{2}} \notag\\
		&\le &  C(\eta) \|v\|_{L_x^\infty }\|v\|_{L_x^2}+ \eta\|\frac{1}{h^2}G_h^w((\xi_k-d\phi_k(x_k))^2)(|v|v)\|_{L_x^2}.\label{8292}
	\end{eqnarray}
Substituting (\ref{6293})--(\ref{8292}) into (\ref{6294}) and   choosing $\eta>0$ sufficiently small, it follows  that
	\begin{equation}
		\|\mathcal{L} _k^2(|v| v)\|_{L_x^2}\lesssim  \|v\|_{L_x^\infty } (\|v\|_{L_x^2}+\|\frac{1}{h^2}G_h^w((\xi_k-d\phi_k(x_k))^2)v\|_{L_x^2}),\qquad k=1,2. \label{8294}
	\end{equation}
Taking the Weyl quantization in (\ref{6284}), then using 	Propositions \ref{pjl2} and \ref{pcomposition}, we obtain
	\begin{eqnarray}
		\|\frac{1}{h^2}G_h^w((\xi_k-d\phi_k(x_k))^2)v\|_{L_x^2} \lesssim \|\mathcal{L} ^2_kv\|_{L_x^2}  +\|\frac{1}{h} G_h^w((\xi_k-d\phi_k(x_k)))v\|_{L_x^2}+\|v\|_{L_x^2}. \label{731}
	\end{eqnarray}
	Moreover, it follows from (\ref{6292})--(\ref{813z1}), Gagliardo-Nirenberg and Young's inequality that
	\begin{eqnarray}
		\|\frac{1}{h} G_h^w((\xi_k-d\phi_k(x_k)))v\|_{L_x^2} &\lesssim&  \| e^{-it\phi_k(x_k)}v\|_{L_x^2} ^{\frac{1}{2}} \|\partial_{x_k x_k} e^{-it\phi_k(x_k)}v\|_{L_x^2} ^{\frac{1}{2}}\notag\\
		&\lesssim &  \|v\|_{L_x^2 } ^{\frac{1}{2}} \|\frac{1}{h^2}G_h^w((\xi_k-d\phi_k(x_k))^2)v\|_{L_x^2}^{\frac{1}{2}} \notag \\
		&\le & \eta \|\frac{1}{h^2}G_h^w((\xi_k-d\phi_k(x_k))^2)v\|_{L_x^2} +C(\eta) \|v\|_{L_x^2}. \label{732}
	\end{eqnarray}
	Combining (\ref{731}) and (\ref{732}),  choosing $\eta>0$ sufficiently small, we obtain
	\begin{eqnarray}
		\|\frac{1}{h^2}G_h^w((\xi_k-d\phi_k(x_k))^2)v\|_{L_x^2} \lesssim \|\mathcal{L} _k^2v\|_{L_x^2} +\|v\|_{L_x^2},\ k=1,2. \label{6295}
	\end{eqnarray}
	The inequality (\ref{lvalpha}) is now an immediate consequence of (\ref{8294}) and (\ref{6295}).
	
	Finally, we prove (\ref{751}).  Taking the Weyl quantization in (\ref{741}), then using  Propositions \ref{pjl2} and \ref{pcomposition}, we obtain
	\begin{eqnarray}
		\|\mathcal{L} _kv\|_{L_x^2} \le \|\frac{1}{h}G_h^w(\xi_k-d\phi_k(x_k))v\|_{L_x^2} +\|v\|_{L^2},\ k=1,2.\label{210w1}
	\end{eqnarray}
	Substitution of  (\ref{732}) and (\ref{6295}) into (\ref{210w1}) yields the desired estimate (\ref{751}). This completes the proof of Lemma \ref{l2v}.
\end{proof}
\section*{Acknowledgements}
This work is
partially supported by  National Natural Science Foundation of China    11931010, and Zhejiang Provincial Natural Science Foundation of China LDQ23A010001.

\end{document}